\documentclass[twoside, 12pt]{amsart}
\usepackage{latexsym}
\usepackage{amssymb,amsmath,amsopn}
\usepackage[dvips]{graphicx}   %To insert ps figures
\usepackage{color,epsfig}      %To insert ps figures
\usepackage{url}

\makeatletter
\def\@settitle{\begin{center}%
  \baselineskip14\p@\relax
  \bfseries
  \uppercasenonmath\@title
  \@title
  \ifx\@subtitle\@empty\else
     \\[1ex]\uppercasenonmath\@subtitle
     \footnotesize\mdseries\@subtitle
  \fi
  \end{center}%
}
\def\subtitle#1{\gdef\@subtitle{#1}}
\def\@subtitle{}
\makeatother

\newtheorem{cor}{Corollary}

\newtheorem{thm}{Theorem}
\newtheorem{rem}{Remark}
\newtheorem{prop}{Proposition}[section]
\newtheorem{lem}{Lemma}

\newtheorem{defn}{Definition}
\newtheorem{prob}{Problem}
\newtheorem{conj}{Conjecture}

\DeclareMathOperator{\ker1}{Ker}

\DeclareMathOperator{\conv}{conv}

\DeclareMathOperator{\vol}{vol}

\DeclareMathOperator{\perim}{perim}

\DeclareMathOperator{\aff}{aff}

\DeclareMathOperator{\surf}{surf}
\DeclareMathOperator{\cirr}{cr}
\DeclareMathOperator{\ir}{ir}
\DeclareMathOperator{\proj}{proj}

\renewcommand{\Re}{\mathbb{R}}

\newcommand{\BB}{\mathbf{B}}

\newcommand{\Sph}{\mathbb{S}}

\renewcommand{\S}{\mathcal{S}}

\newcommand{\zl}{\color{black}}

\begin{document}

\title[Isoperimetric problems for zonotopes]{Isoperimetric problems for zonotopes}

\author[A. Jo\'os]{Antal Jo\'os}
\author[Z. L\'angi]{Zsolt L\'angi}

\address{Antal Jo\'os\\ Department of Mathematics, University of Duna\'ujv\'aros, T\'ancsics M. u. 1/a, Duna\'ujv\'aros, Hungary, 2400}
\email{joosa@uniduna.hu}
\address{Zsolt L\'angi\\ Department of Geometry, Budapest University of Technology, Egry J\'{o}zsef utca 1., Budapest, Hungary, 1111 and\\
MTA-BME Morphodynamics Research Group, M\H uegyetem rkp. 3., Budapest, Hungary, 1111}
\email{zlangi@math.bme.hu}

\thanks{Partially supported by the BME Water Sciences \& Disaster Prevention TKP2020 Institution Excellence Subprogram, grant no. TKP2020 BME-IKA-VIZ and
the NKFIH grant K134199.}


\begin{abstract}
Shephard (Canad. J. Math. 26: 302-321, 1974) proved a decomposition theorem for zonotopes yielding a simple formula for {\zl their volume}. In this note we prove a generalization of this theorem yielding similar formulas for {\zl their intrinsic volumes}. We use this result to investigate geometric extremum problems for zonotopes generated by a given number of segments. In particular, we solve isoperimetric problems for $d$-dimensional zonotopes generated by $d$ or $d+1$ segments, and give asymptotic estimates for the solutions of similar problems for zonotopes generated by sufficiently many segments. In addition, we present applications of our results to the $\ell_1$ polarization problem on the unit sphere and to a vector-valued Maclaurin inequality conjectured by Brazitikos and McIntyre in 2021.
\end{abstract}

\maketitle

\section{Introduction}\label{sec:intro}

An important family of convex polytopes in the $d$-dimensional Euclidean space $\Re^d$ is the class of so-called \emph{zonotopes}, polytopes obtained as Minkowski sums of finitely many segments. Zonotopes have been in the focus of research since the middle of the 20th century, earning a separate chapter in the famous book
\cite{Coxeter} of Coxeter in 1943 and a place in the problem collection \cite{Shephard_problems} of Shephard. They are connected to various branches of mathematics, as an example, we may mention the fact that their face lattices correspond to the combinatorial classes of central hyperplane arrangements in $\Re^{d+1}$ \cite{Grunbaum, winter}. They play a central role in the theory of projections of convex polytopes, both being the projections of affine cubes (see e.g. \cite{McMullen2}), and, by Cauchy's projection formula \cite{Gardner}, being the projection bodies of convex polytopes in $\Re^d$ (see also \cite{parallelohedra}). They are closely related to parallelohedra, the convex polytopes whose translates fill the space (see \cite{Erdahl, McMullen3}), and appear in lattice geometry (see e.g. \cite{Lenz, HM17, BJM19} or the survey \cite{survey}). Zonotopes are also investigated and often applied outside pure mathematics \cite{ABC05, BernEppstein, BEGHSW, GNZ}.

We note that there is a rich literature about isoperimetric type problems for zonotopes. From amongst these papers, we mention the paper \cite{Bezdek} of Bezdek comparing the inradius of a rhombic dodecahedron to its intrinsic volumes, \cite{Filliman} of Filliman comparing the volume of a zonotope to the total squared lengths of its generating vectors, and \cite{Deza} of Deza, Pournin and Sukegawa giving a sharp asymptotic estimate on the maximum diameter of a primitive zonotope. Recently, Fradelizi et al. \cite{FMMZ} investigated various volume inequalities for the Minkowski sums of zonoids. For open problems regarding zonotopes, the interested reader is referred also to \cite{Ivanov}.
% We remark that the results \cite{Linhart} of Linhart and \cite{BourgLind} of Bourgain and Lindenstrauss can be regarded as estimates for the ratio of the circumradius to the inradius of a zonotope, generated by a given number of segments. The latter problem was investigated in \cite{BLM} by Bourgain, Lindenstrauss and Milman in a more general form in which an arbitrary zonoid plays the role of the Euclidean ball.

{\zl An elegant, simple formula exists for the volume of a zonotope in terms of its generating segments; this formula is usually attributed to McMullen~\cite{McMullen} or Shephard \cite{Shephard} and follows e.g. from a decomposition theorem for zonotopes in \cite{Shephard}. In \cite{GoverKrikorian}, this formula was extended to zonotopes whose dimension is strictly less than that of the ambient space. Using an integral geometric approach, in a very recent paper Brazitikos and McIntyre \cite{BM} found a generalization of this formula for all intrinsic volumes of a zonotope.}

As a preliminary result, we prove a generalization of {\zl Shephard's decomposition} theorem {\zl reproving the formulas for the intrinsic volumes of a zonotope in \cite{BM}.} Our proof is presented in Section~\ref{sec:prelim}, where we also collect all the tools required in our investigation.
In the main part of the paper we prove isoperimetric inequalities for zonotopes. We focus on two types of zonotopes. In particular, we investigate zonotopes in $\Re^d$ generated by $d$ or $d+1$ segments, and also zonotopes in $\Re^d$ generated by $n \gg d$ segments. We collect our results for these two types of zonotopes in Sections~\ref{sec:rhombic} and \ref{sec:asymptotic}, respectively.

{\zl In the paper, we also point out} applications of our results to other problems in mathematics.
{\zl We focus on two problems, one of which is the so-called \emph{$\ell_p$-polarization problem} discussed e.g. in \cite{AN} (see also \cite{NR}), and the other one is a generalization of the well known Maclaurin inequality \cite{Biler} by Brazitikos and McIntyre \cite{BM}, comparing the values of elementary symmetric functions evaluated at positive real numbers. We introduce these problems and their relations to zonotopes in detail in Section~\ref{sec:prelim}, and describe our related results at their appropriate places in Sections~\ref{sec:rhombic} and \ref{sec:asymptotic}.}

Finally, in Section~\ref{sec:problems} we collect some remarks and open questions.

\section{Preliminaries}\label{sec:prelim}

In our investigation, we denote the origin of the space $\Re^n$ by $o$. We denote the closed unit ball centered at $o$ by $\BB^d$, and its boundary and volume by $\Sph^{d-1}$ and $\kappa_d$, respectively, where we set $\kappa_0 = 1$. For any points $p,q \in \Re^d$, we denote the closed segment with endpoints $p,q$ by $[p,q]$, and the Euclidean norm of $p$ by $|p|$. {\zl We call a compact, convex set with non-empty interior a \emph{convex body}.}

{\zl Let $K \subset \Re^d$ be a compact, convex set. By Steiner's formula, the volume of the convex body $K+ t \BB^d$ is a polynomial of $t$ of degree $d$. We write this polynomial in the form
\begin{equation}\label{eq:Steiner}
V_d(K+ t \BB^d) = \sum_{i=0}^d \kappa_{d-i} V_i(K) t^{d-i}.
\end{equation}
For $i=1,2,\ldots, n$, the quantity $V_i(K)$ is called the \emph{$i$th intrinsic volume of $K$} \cite{Gardner}.} We recall that if the dimension of $K$ is $m$ with $1 \leq m \leq d$, then for any $i > m$, $V_i(K)=0$, and $V_m(K)$ is equal to the $m$-dimensional volume (Lebesgue measure) of $K$. Finally, if $K \subset \Re^d$ is a convex body, we denote the surface area and the mean width of $K$ by $\surf(K)$ and w(K), respectively.

It is a well known fact that for any convex body $K \subset \Re^d$ there is a unique ball of minimal radius containing $K$. We call this ball the \emph{circumball} of $K$, its radius the \emph{circumradius} of $K$, which we denote by $\cirr(K)$. Similarly, we call a largest ball contained in $K$ an \emph{inball} of $K$, and denote its radius by $\ir(K)$. Finally, for any $v_1, v_2, \ldots, v_d \in \Re^d$, we denote the $d \times d$ matrix with the $v_i$ as column vectors by $[v_1, v_2, \ldots, v_n]$. We remark that since every zonotope $Z$ is centrally symmetric, the quantities $2\cirr(Z)$ and $2\ir(Z)$ are equal to the diameter and the minimum width of $Z$.

In this paper we distinguish zonotopes satisfying some additional property. More specifically, if the lengths of all generating vectors of a zonotope $Z$ are equal (resp., equal to $1$), we say that $Z$ is an \emph{equilateral} (resp., \emph{unit edge length}) zonotope. Note that up to translation, any zonotope $Z$ in $\Re^d$ can be defined as $Z= \sum_{i=1}^n [o,p_i]$ for some $p_1, p_2, \ldots, p_n \in \Re^d$. In this case we say that $Z= \sum_{i=1}^n [o,p_i]$ is given in a \emph{canonical form}. Clearly, the zonotope $Z= \sum_{i=1}^n [o,p_i]$ is symmetric to the origin $o$ if and only if $\sum_{i=1}^n p_i = o$; in this case we say that $Z$ has a \emph{centered canonical form}. We denote the family of zonotopes in $\Re^d$ generated by at most $n$ segments by $\mathcal{Z}_{d,n}$.

We note that the zonotope $Z=\sum_{i=1}^n [o,p_i]$, $p_i \in \Re^d$ is the translate of the zonotope $Z'= \frac{1}{2} \sum_{i=1}^n [-p_i,p_i]$. Thus, since $Z'$ is $o$-symmetric and every vertex of $Z'$ is of the form $\sum_{i=1}^n \varepsilon_i p_i$ with some $\varepsilon_i \in \{ -1,1\}$, we have that
\begin{equation}\label{eq:cirr}
\cirr(Z) = \frac{1}{2} \max \left\{  \left|\sum_{i=1}^n \varepsilon_i p_i \right| : \varepsilon_i \in \{ -1,1\}, i=1,2,\ldots, n \right\}.
\end{equation}
Recall that the support function of a convex body $K \subset \Re^d$ is defined as $h_K : \Sph^{d-1} \to \Re$, $h_K(u) = \sup \{ \langle x,u \rangle : x \in K \}$. Furthermore, recall that for any convex bodies $K,L \subset \Re^d$, we have $h_{K+L} = h_K + h_L$, and $h_{[-q,q]}(u) = | \langle q,u \rangle |$ for any $q \in \Re^n$.
As $\cirr(K) = \max \{ h_K(u) : u \in \Sph^{d-1} \}$ and $\ir(K) = \min \{ h_K(u) : u \in \Sph^{d-1} \}$, this yields that for $Z=\sum_{i=1}^n [o,p_i]$, $p_i \in \Re^d$,
\begin{equation}\label{eq:cirr_ir}
\cirr(Z) = \frac{1}{2} \max \left\{ \sum_{i=1}^n | \langle u,p_i \rangle  | : u \in \Sph^{d-1} \right\}, \hbox{ and } \ir(Z) = \frac{1}{2} \min \left\{ \sum_{i=1}^n | \langle u,p_i \rangle  | : u \in \Sph^{d-1} \right\} .
\end{equation}

%An elegant, simple formula exists for the volume of a zonotope in terms of its generating vectors; this formula is usually attributed to McMullen~\cite{McMullen} or Shephard \cite{Shephard} and follows e.g. from a decomposition theorem for zonotopes in \cite{Shephard}. In \cite{GoverKrikorian}, this formula was extended to zonotopes whose dimension is strictly less than that of the ambient space. Using an integral geometric approach, Brazitikos and McIntyre \cite{BM} found a generalization of this formula for all intrinsic volumes of a zonotope.
Our first goal is to prove a generalization of the decomposition theorem in \cite{Shephard}, which yields an elementary proof of the formula in \cite{BM}. To state this result, we need some preparation.

Let $d \geq 2$, and consider a zonotope $Z = \sum_{i=1}^n [o,p_i]$ with $p_i \in \Re^d$ for all values of $i$.
We set $\mathcal{P}^Z=\{ p_1, \ldots, p_n \}$, and for any $0 \leq k \leq d$, by $\mathcal{P}^Z_k$ the family of $k$-element subsets of $\mathcal{P}^Z$ containing linearly independent vectors.
Furthermore, for any $I \in \mathcal{P}^Z_k$ with $k \geq 1$, we set $P(I) = \sum_{i \in I} [o,p_i]$, and for later use, we extend this notation for the case $k=0$ by setting $P(\emptyset) = \{ o \}$, which we regard as a $0$-dimensional parallelotope. Note that for any $I = \{ i_1, i_2, \ldots, i_k \} \in \mathcal{P}^Z_k$, the $k$-dimensional volume of $P(I)$ is $V_k(P(I))$, which is equal to the length of the wedge product of the vectors $p_{i_j}$ {\zl \cite{BM}}, i.e. $V_k(P(I)) = | p_{i_1} \wedge p_{i_2} \wedge \ldots \wedge p_{i_k}|$.
In addition, for any parallelotope $P(I)$ with $I \in \mathcal{P}^Z_k$, we set $B^{\perp}(I) = \BB^d \cap \left( \aff P(I) \right)^{\perp}$ and $\S^{\perp}(I) = \Sph^{d-1} \cap \left( \aff P(I) \right)^{\perp}$, where $\left( \aff P(I) \right)^{\perp}$ denotes the orthogonal complement of $\aff P(I)$.
Finally, by a decomposition of a convex body $K$ we mean is a finite family of mutually non-overlapping convex bodies whose union is $K$.

\begin{thm}\label{thm:main1}
{Using the notation in the previous paragraph, for} any $t \geq 0$, the set $Z + t \BB^d$ can be decomposed into a family $\mathcal{F}_Z$ of mutually non-overlapping convex bodies of the form $X + t B_X$ such that
\begin{enumerate}
\item for any $X + t B_X \in \mathcal{F}_Z$, $X$ is a translate of some parallelotope $P(I)$ with $I \in \mathcal{P}^Z_k$ for some $0 \leq k \leq d$, and $B_X \subseteq B^{\perp}(I)$ is the convex hull of $o$ and a spherically convex, compact subset of $S^{\perp}(I)$;
\item if for any $0 \leq k \leq d$ and $I \in \mathcal{P}^Z_k$, $\mathcal{F}_Z(I)$ denotes the subfamily of the elements $X + t B_X$ of $\mathcal{F}_Z$, where $X$ is a translate of $P(I)$, then $\{ B_X : X+t B_X \in \mathcal{F}_Z(I) \}$ is a decomposition of $B^{\perp}(I)$.
\end{enumerate}
\end{thm}

\begin{figure}[ht]
  \centering
  \includegraphics[width=8cm]{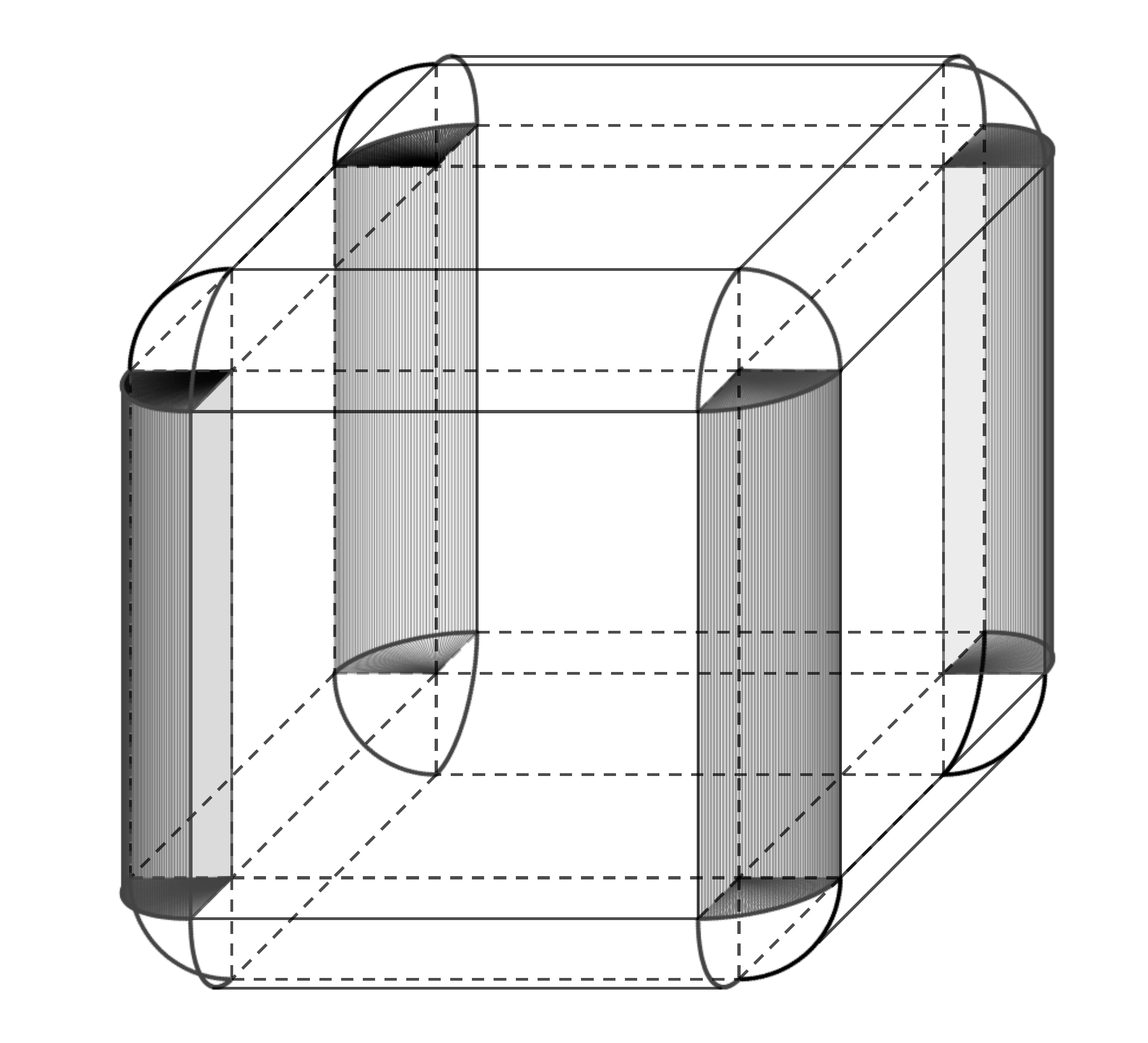}
  \caption{The body $Z + t \BB^d$ if $Z$ is a cube generated by $3$ mutually orthogonal segments. There are $4$ translates of every generating segment appearing as edges {\zl of $Z$}. The solid bodies in the picture correspond to the sets $X + t B_X$, where $X$ is a translate of a fixed generating segment.}
\end{figure}

We note that the case $t=0$ of Theorem~\ref{thm:main1} coincides with the decomposition theorem in \cite{Shephard}.

\begin{proof}
Recall \cite{Shephard} that a zonotope $Z$ is called \emph{cubical} if every facet of $Z$ is an affine cube, or equivalently, if any at most $d$ of its generating vectors in a canonical form are linearly independent. We prove the statement only for cubical zonotopes. In particular, we show that in this case Theorem~\ref{thm:main1} holds with the additional assumption that if $X + t B_X \in \mathcal{F}_Z$ with $\dim X < d$, then $X$ is a face of $Z$ and $B_X$ is the set of the outer normal vectors of $X$ of length at most one, and observe that the general case follows from this by a standard limit argument.

We prove the theorem by induction on the dimension $d$. Observe that for $d=1$ the statement is trivial. Assume that it holds for any zonotope of dimension strictly less than $d$, and let $Z=\sum_{i=1}^n [o,p_i]$, $p_i \in \Re^n$ be a zonotope.
Let $F$ be a $k$-face of $Z$. Then there is a supporting hyperplane $H$ of $Z$ with $F = H \cap Z$. By the definition of a zonotope, $F$ is a translate of the generating segments $[o,p_i]$ with the property that $p_i$ is parallel to $H$. Thus, since $Z$ is cubical, any $k$-face of $Z$ is a translate of some $P(I)$ with $I \in \mathcal{P}^Z_k$. Let $\proj : \Re^n \to \left( \aff (P(I)) \right)^{\perp}$ be the orthogonal projection onto $\left( \aff (P(I)) \right)^{\perp}$.
Then $\proj(Z)$ is a zonotope generated by the projections of the generating segments of $Z$, and since $Z$ is cubical, the vertices of $\proj(Z)$ coincide with the projections of the $k$-faces of $Z$ that are translates of $P(I)$. Since it is well known that the sets of the outer unit normal vectors of a convex polytope in $\Re^{d-k}$ is a decomposition of $\Sph^{d-k-1}$ into spherically convex sets {\zl (see, e.g. \cite{Schneider})}, we proved that the closure of $(Z + t \BB^d) \setminus Z$ can be decomposed with sets $X+tB_X$, where $X$ is a proper face of $Z$, $B_X$ is the set of outer normal vectors of $X$ of length at most one, and for any $I \in \mathcal{P}^Z_k$, $0 \leq k < d$,
$\{ B_X : X+t B_X \in \mathcal{F}_Z(I) \}$ is a decomposition of $B^{\perp}(I)$.

We are left to show that $Z$ can be decomposed into translates of $P(I)$ with $I \in \mathcal{P}^Z_d$ such that for any $I \in \mathcal{P}^Z_d$, there is exactly one translate of $P(I)$ in the decomposition. We prove it by induction on $n$. If $n=d$, the statement is obvious. Assume that it holds for cubical zonotopes generated by at most $n-1$ segments. Let $W = \sum_{i=1}^{n-1} [o,p_i]$. Then $W$ can be decomposed into translates of $P(I)$ with $I \in \mathcal{P}^W_d$ according to the requirements. Furthermore, by the previous paragraph, for any $J \in \mathcal{P}^W_{d-1}$ there are exactly two facets of $W$ and $Z$ that are translates of $P(J)$. By the definition of a zonotope, if $F_1$ and $F_2$ are the above two facets of $W$, then either $F_1$ and $p_n+F_2$, or $F_2$ and $p_n+F_1$ are facets of $Z$. Thus, the closure of $Z \setminus W$ can be decomposed into translates of $P(J \cup \{ n \})$, where $J \in \mathcal{P}^W_{d-1}$, such that for each $J \in \mathcal{P}^W_{d-1}$, there is exactly one translate of $P(J \cup \{ n \})$ in the decomposition. This proves Theorem~\ref{thm:main1}.
% for cubical zonotopes. Finally, for the general case, we may apply a standard limit argument.
\end{proof}

For Corollary~\ref{cor:intrinsic_formula}, see also \cite{BM}.

\begin{cor}\label{cor:intrinsic_formula}
For any zonotope $Z = \sum_{i=1}^n [o,p_i]$ in $\Re^d$, and any $0 \leq k \leq d$, the $k$th intrinsic volume of $Z$ is
\begin{equation}\label{eq:intrinsic}
V_k(Z) = \sum_{1 \leq i_1 < i_2 < \ldots < i_k \leq n} | p_{i_1} \wedge p_{i_2} \wedge \ldots \wedge p_{i_k} |.
\end{equation}
\end{cor}

\begin{proof}
Theorem~\ref{thm:main1} yields that for any $t > 0$
\[
V_d(Z+t \BB^d) = \kappa_d t^d + \sum_{k=1}^d \sum_{1 \leq i_1 < i_2 < \ldots < i_k \leq n} | p_{i_1} \wedge p_{i_2} \wedge \ldots \wedge p_{i_k} | \kappa_{d-k} t^{d-k}.
\]
{\zl Corollary~\ref{cor:intrinsic_formula} readily follows from comparing this formula with Steiner's formula in (\ref{eq:Steiner}).}
\end{proof}

\begin{rem}\label{rem:meanwidth}
Recall that the mean width $w(K)$ of a convex body $K \subset \Re^d$ is $w(K) = \frac{2 \kappa_{d-1}}{d \kappa_d} V_1(K)$ (see \cite{Gardner}). Thus, the mean width of a zonotope $Z=\sum_{i=1}^n [o,p_i]$ in $\Re^d$ is $w(Z) = \frac{2 \kappa_{d-1}}{d \kappa_d} \sum_{i=1}^n  |p_i|$.
\end{rem}

\begin{rem}\label{rem:projectionbody}
Let $P$ be a convex polytope in $\Re^d$ with $m$ facets such that the outer unit normal vector of the $i$th facet of $P$ is $u_i$ and its $(d-1)$-dimensional volume is $F_i$. {\zl Recall that for a convex body $K \in \Re^n$, the function $h : \Sph^{d-1} \to \Re$, $h(u) = V_{d-1}(K | x^{\perp})$, where $K | x^{\perp}$ is the orthogonal projection of $K$ onto the hyperplane through $o$ and with unit normal vector $x$, is the support function of a convex body; this body, denoted by $\Pi K$, is called the \emph{projection body} of $K$ \cite{Gardner}. Cauchy's projection formula \cite{Gardner} states that $V_{d-1}(K|x^{\perp}) = \int_{\Sph^{d-1}} |\langle x,u \rangle| \, dS(K,u)$, where $S(K,\cdot)$ is the surface area measure of $K$. Applying this for $P$, we obtain that $h_{\Pi K} (x) = \sum_{i=1}^m F_i |\langle x, u_i \rangle |$ for any $x \in \Sph^{d-1}$. Thus, the additivity of the support function with respect to Minkowski sums of convex bodies yields that
\[
\Pi P = \sum_{i=1}^m [o,F_i u_i]
\]
In particular, by Corollary~\ref{cor:intrinsic_formula} and since both mean width and surface area are continuous functions with respect to Hausdorff distance, for any convex body $K \subset \Re^d$, we have $w(\Pi K) = \frac{2 \kappa_{d-1}}{d \kappa_d} \surf (K)$.}
\end{rem}

{\zl
In the remaining part of Section~\ref{sec:prelim}, we introduce two problems that are related to isoperimetric problems for zonotopes.

For $p > 0$ and a multiset $\omega_n=\{ x_1, x_2, \ldots, x_n \}$ in $\Sph^{d-1}$, the \emph{$\ell_p$-polarization of $\omega_n$} is defined as
\[
M_p(\omega_n) = \max \left\{ \sum_{i=1}^n \left| \langle x_i, u \rangle \right|^p : u \in \Sph^{d-1} \right\},
\]
and the quantity
\[
M_n^p(\Sph^{d-1}) = \min \left\{ M_p(\omega_n) : \omega_n \subset \Sph^{d-1} \right\}
\]
is called the \emph{$\ell_p$-polarization (or Chebyshev) constant of $\Sph^{d-1}$}.
The $\ell_p$-polarization problem on the sphere asks for determining the value of $M_n^p(\Sph^{d-1})$ for all values of $p$ and $d$.

Note that by (\ref{eq:cirr}) and (\ref{eq:cirr_ir}), if $Z=\sum_{i=1}^n [o,x_i]$ is a zonotope in $\Re^d$, then
\begin{equation}\label{eqcr1}
\cirr(Z) = \frac{1}{2} \max \left\{  \left|\sum_{i=1}^n \varepsilon_i x_i \right| : \varepsilon_i \in \{ -1,1\}, i=1,2,\ldots, n \right\},
\end{equation}
or equivalently (using the definition of support function),
\begin{equation}\label{eqcr2}
\cirr(Z) = \frac{1}{2} \max \left\{ \sum_{i=1}^n | \langle u,x_i \rangle  | : u \in \Sph^{d-1} \right\}.
\end{equation}
Thus, the $\ell_1$-polarization constant of $\Sph^{d-1}$ is equal to twice the minimal circumradius of a zonotope generated by $n$ segments of unit length. By Remark~\ref{rem:meanwidth}, the mean width of a zonotope is a scalar multiple of the total length of its generating vectors. Thus, this problem can be restated as finding the minimum circumradius of a zonotope generated by $n$ segments of unit length. In particular, for the special case of unit vectors, the equality of the expressions in (\ref{eqcr1}) and (\ref{eqcr2}) is proved as Proposition 3 of \cite{AN} by the Lagrange multiplier method.

The function $\sigma^k_m(x_1,\ldots, x_m) = \sum_{1 \leq i_1 < i_2 < \ldots < i_k \leq m} x_{i_1} x_{i_2} \ldots x_{i_k}$ {\zl is called} the $k$th elementary symmetric function on the $m$ variables $x_1,\ldots,x_m$. The following statement, called Maclaurin's inequality, can be found e.g. in \cite{Biler}.

\begin{lem}[Maclaurin's inequality]\label{lem:Maclaurin}
Let $1 \leq k < m$ be integers, and $x_1, \ldots, x_m > 0$ be positive real numbers. Then
\[
\left( \frac{\sigma_m^k(x_1,x_2,\ldots,x_m)}{\binom{m}{k}} \right)^{\frac{1}{k}} \geq \left( \frac{\sigma_m^\textbf{k+1}(x_1,x_2,\ldots,x_m)}{\binom{m}{k+1}} \right)^{\frac{1}{k+1}},
\]
with equality if and only if all $x_i$ are equal.
\end{lem}

A generalization of this inequality for a set of vectors in $\Re^d$ was conjectured in \cite{BM} as follows.

\begin{conj}[Brazitikos, McIntyre]\label{conj:Maclaurin}
Let $x_1, x_2, \ldots, x_n \in \Re^d$ be given with $1 \leq d \leq n$. Then for any $p \in [0,\infty]$ and $2 \leq k \leq d$, we have
\begin{equation}\label{eq:vv_Mac}
\left( \frac{\sum_{1 \leq i_1 < \ldots < i_k \leq n} |x_{i_1} \wedge x_{i_2} \wedge \ldots \wedge x_{i_k} |^p}{ \binom{n}{k}}\right)^{\frac{1}{p k}} \leq
\left( \frac{\sum_{1 \leq i_1 < \ldots < i_{k-1} \leq n} |x_{i_1} \wedge x_{i_2} \wedge \ldots \wedge x_{i_{k-1}} |^p}{ \binom{n}{k-1}}\right)^{\frac{1}{p (k-1)}},
\end{equation}
with equality if and only if $n=d$ and the vectors form an orthonormal basis.
\end{conj}

The authors of \cite{BM} proved this conjecture for $p=0$ and $p=\infty$, for $p=2$ and $n=d$, and for $p=1$, $n=d$ and $k=2,3,d$. They also pointed out (see also Corollary~\ref{cor:intrinsic_formula}) that if $p=1$, the numerators in (\ref{eq:vv_Mac}) correspond to the $k$th and $(k-1)$st intrinsic volumes of the zonotope $\sum_{i=1}^n [o,x_i]$, showing that this case of Conjecture~\ref{conj:Maclaurin} directly leads to isoperimetric problems for zonotopes.
}

\section{Isoperimetric problems for zonotopes generated by a small number of segments}\label{sec:rhombic}

Zonotopes in $\mathcal{Z}_{d,d}$ are called parallelotopes, and those in $\mathcal{Z}_{d,d+1}$ are called rhombic dodecahedra \cite{Bezdek}. The goal of this section is to prove isoperimetric inequalities for them. We note that if $Z=\sum_{i=1}^{d+1} [o,p_i]$ is a rhombic dodecahedron where the points $p_i$ are the vertices of a regular simplex centered at $o$, then $Z$ is called a \emph{regular rhombic dodecahedron}. A regular rhombic dodecahedron in $\Re^3$ is the Voronoi cell of a face-centered cubic lattice. Thus, by the seminal result of Hales \cite{Hales} proving Kepler's conjecture, among rhombic dodecahedra in $\Re^3$ with unit inradius, the ones with minimal volume are the regular ones. This consequence of the result of Hales was strengthened in \cite{Bezdek}, which we quote for completeness.

\begin{thm}[Bezdek]\label{thm:Bezdek}
Let $1 \leq k \leq d$ be arbitrary. Among rhombic dodecahedra in $\Re^d$ of unit inradius, the ones with minimal $k$th intrinsic volumes are the regular ones.
\end{thm}

Intuitively, regular rhombic dodecahedra appear as natural candidates for solutions in isoperimetric problems for rhombic dodecahedra. Nevertheless, as we will see in Subsection~\ref{subsec:rh_cr_mw}, this does not always hold.

In Subsection~\ref{subsec:rh_vol} we find the minimal intrinsic volumes and circumradii of the elements of $\mathcal{Z}_{d,d}$ and $\mathcal{Z}_{d,d+1}$ with a given volume. In Subsections~\ref{subsec:rh_cr_mw} and \ref{subsec:rh_mw_V2} we investigate the problems of finding the minimal circumradius and the second intrinsic volume $V_2(\cdot)$, respectively, of a parallelotope or a rhombic dodecahedron of a given mean width, respectively.

In Subsection~\ref{subsec:rh_squared} our investigation is motivated by a result of Tanner \cite{Tanner}, who defined the \emph{total squared $k$-content of a simplex $S \subset \Re^n$} as the sum of the squares of the $k$-volumes of its $k$-faces, and proved that among simplices with a given squared $1$-content, regular simplices have maximal squared $k$-content for all $1 \leq k \leq d$. In this subsection we define the total squared $k$-volume of a zonotope with $1 \leq k \leq d$, and find its minimum among the parallelotopes and rhombic dodecahedra with a given squared $l$-volume for all $k < l \leq d$.

In our investigation, we set $\mathcal{Z}_{\mathrm{p}} = \mathcal{Z}_{d,d}$ and $\mathcal{Z}_{\mathrm{rd}} = \mathcal{Z}_{d,d+1}$, and we denote by $Z_{\mathrm{p}}^{\mathrm{reg}} \in  \mathcal{Z}_{d,d}$ a cube, and by $Z_{\mathrm{rd}}^{\mathrm{reg}} \in \mathcal{Z}_{d,d+1}$ a regular rhombic dodecahedron.

\subsection{The intrinsic volumes and circumradius of a zonotope with a given volume}\label{subsec:rh_vol}

Our main result in Subsection~\ref{subsec:rh_vol} is Theorem~\ref{thm:volume}.

\begin{thm}\label{thm:volume}
Let $1 \leq k \leq d-1$. Then, for any $Z_i \in \mathcal{Z}_i$ with $V_d(Z_i)=V_d(Z_i^{\mathrm{reg}})$, where $i \in \{ \mathrm{p}, \mathrm{rd} \}$, we have
\[
V_k(Z_i) \geq V_k(Z_i^{\mathrm{reg}})
\]
with equality if and only if $Z_i$ is congruent to $Z_i^{\mathrm{reg}}$. Furthermore, we have
\[
\cirr(Z_i) \geq \cirr(Z_i^{\mathrm{reg}}).
\]
\end{thm}

We note that our next corollary readily follows from Theorems~\ref{thm:Bezdek} and \ref{thm:volume}.

\begin{cor}\label{cor:ircr}
For any $Z_i \in \mathcal{Z}_i$ with $\ir(Z_i)=\ir(Z_i^{\mathrm{reg}})$, where $i \in \{ \mathrm{p}, \mathrm{rd} \}$, we have
\[
\cirr(Z_i) \geq \cirr(Z_i^{\mathrm{reg}})
\]
with equality if and only if $Z_i$ is congruent to $Z_i^{\mathrm{reg}}$.
\end{cor}

We prove Theorem~\ref{thm:volume} only for $i=\mathrm{rd}$, as the same argument can be used to prove it for $i=\mathrm{p}$. We note that the case $i=\mathrm{p}$ and $k=d-1$ of our theorem was also proved in Theorem 5 of \cite{BM}. Our proof is based on an application of shadow systems \cite{Saroglou} (also called linear parameter systems \cite{RogersShephard}) and Steiner symmetrization.
We start with two lemmas needed for the proof.

\begin{lem}\label{lem:shadow}
Let {\zl $k \geq 2$,} $p_1, \ldots, p_k, v \in \Re^d$, and $\lambda_1, \ldots, \lambda_k \in \Re$. For $i=1,2,\ldots,k$ and all $t \in \Re$, set $p_i(t) = p_i + \lambda_i t v$ and
$S(t) = \conv \{ p_1(t), \ldots, p_k(t)  \}$.
Then the function $f: \Re \to \Re$, $f(t) = \vol_{k-1}(S(t))$ is a convex function of $t$, and if the points $p_1, \ldots, p_k, v$ are affinely independent, then $f$ is strictly convex.
\end{lem}

\begin{proof}
Since $\bigcup_{t \in \Re} S(t)$ is contained in a $k$-dimensional affine subspace of $\Re^d$, without loss of generality, we may assume that $k=d$.
Furthermore, since volume is invariant under translations, {\zl $V_d(S(t))=V_d(S(t)-p_d(t))$, and thus,} we may also assume that $p_d=o$ and $\lambda_d = 0$.
Define the linear functional $L_t : \Re^d \to \Re$, $L_t(x) = \det ([ p_1(t), \ldots, p_{d-1}(t), x])$. This functional can be written in the form
$L_t(x) = \langle u(t), x \rangle$, where $u(t)$ is the vector whose entries are the (signed) minors of size $(d-1) \times (d-1)$ of the $d \times (d-1)$ matrix $[ p_1(t), \ldots, p_{d-1}(t)]$.
Note that by the properties of determinants, $u(t)$ is a linear function of $t$, and hence, it can be written as $u(t) = u+t w$ for some $u,w \in \Re^n$.
Furthermore, $f(t) = \vol_{d-1}(S(t)) = \frac{1}{(d-1)!} | u(t)|$. Thus,
\[
f''(t) = \frac{\langle w,w \rangle \langle u(t), u(t) \rangle - \langle w, u(t) \rangle^2}{(d-1)!|u(t)|^3} \geq 0
\]
by the Cauchy-Schwartz Inequality. This implies the convexity of $f$.

Now, assume that $p_1, \ldots, p_k, v$ are affinely independent. Then, as $p_d=o$, it follows that $p_1, \ldots, p_{d-1}, v$ are linearly independent.
Note that $u(t) \neq o$ for any $t \in \Re$. Indeed, if $u(t)=o$ for some value of $t$, then, by the definition of $L_t(x)$, $p_1(t), \ldots, p_{d-1}(t)$ are linearly dependent for some value of $t$, which yields the existence of a nontrivial linear combination of $p_1, \ldots, p_{d-1}, v$ equal to $o$.
Thus, we have $u(t) \neq o$ for any $t \in \Re$, implying that $u$ is not a scalar multiple of $w$. We show that $w \neq o$. Indeed, suppose for contradiction that $w=o$. Then the value of $L_t(x)$ is independent of $t$ for any value of $x$. In particular, the kernel $\ker1(L_t)$ of $L_t$ is independent of $t$, implying that $p_1, \ldots, p_{d-1}, v \in \ker1(L_t)$. Thus, combined with the fact that $L_t$ is not identically zero, this yields that $p_1, \ldots, p_{d-1},v$ are linearly dependent, contradicting our assumptions. Now, we have shown that $w \neq o$, from which we also have that $u(t)$ and $w$ are not scalar multiples of each other for any value of $t$. From this, the Cauchy-Schwartz Inequality implies that $f''(t) > 0$ for any value of $t$, and hence $f$ is strictly convex.
\end{proof}

\begin{lem}\label{lem:Steiner}
For any simplex $S \subset \Re^d$ and $1 \leq k \leq d-1$, let $\mathcal{F}_k(S)$ denote the family of the $(k-1)$-faces of $S$, and set $f_k(S)= \sum_{F \in \mathcal{F}_k(S)} V_k(\{o \} \cup F)$. Furthermore, if $S= \conv \{ p_1, p_2, \ldots, p_{d+1} \}$, set
\[
g(S) = \max \left\{  \left| \sum_{i=1}^{d+1} \varepsilon_i p_i \right| : \varepsilon_i \in \{ -1,1\}, i=1,2,\ldots, d+1 \right\}.
\]
Let $S_{\mathrm{reg}}$ be a regular simplex centered at $o$. Then, for any $1 \leq k \leq d-1$ and any simplex $S \subset \Re^d$ with $V_d(S) = V_d(S_{\mathrm{reg}})$, we have
\[
f_k(S) \geq f_k(S_{\mathrm{reg}})
\]
with equality if and only if $S$ is a regular simplex centered at $o$. Furthermore, we also have
\[
g(S) \geq g(S_{\mathrm{reg}}).
\]
\end{lem}

\begin{proof}
{\zl Without loss of generality, let $V_d(S_{\mathrm{reg}}) = 1$.}
Clearly, the function $f_k(\cdot)$ attains its minimum value in the family of simplices with {\zl unit} volume. Thus, it is sufficient to show that if $f_k$ is minimal at a simplex $S$ with {\zl $V_d(S) = 1$}, then it is a regular simplex centered at $o$.

Consider a simplex $S$ with $V_d(S) = 1$, which is not a regular simplex centered at $o$.
Let $S= \conv \{p_1, p_2, \ldots, p_{d+1} \}$, and let $H$ be the unique hyperplane perpendicular to the line through the edge $E=[p_1,p_2]$ and passing through $o$. \begin{figure}
  \centering
  \includegraphics[width=10cm]{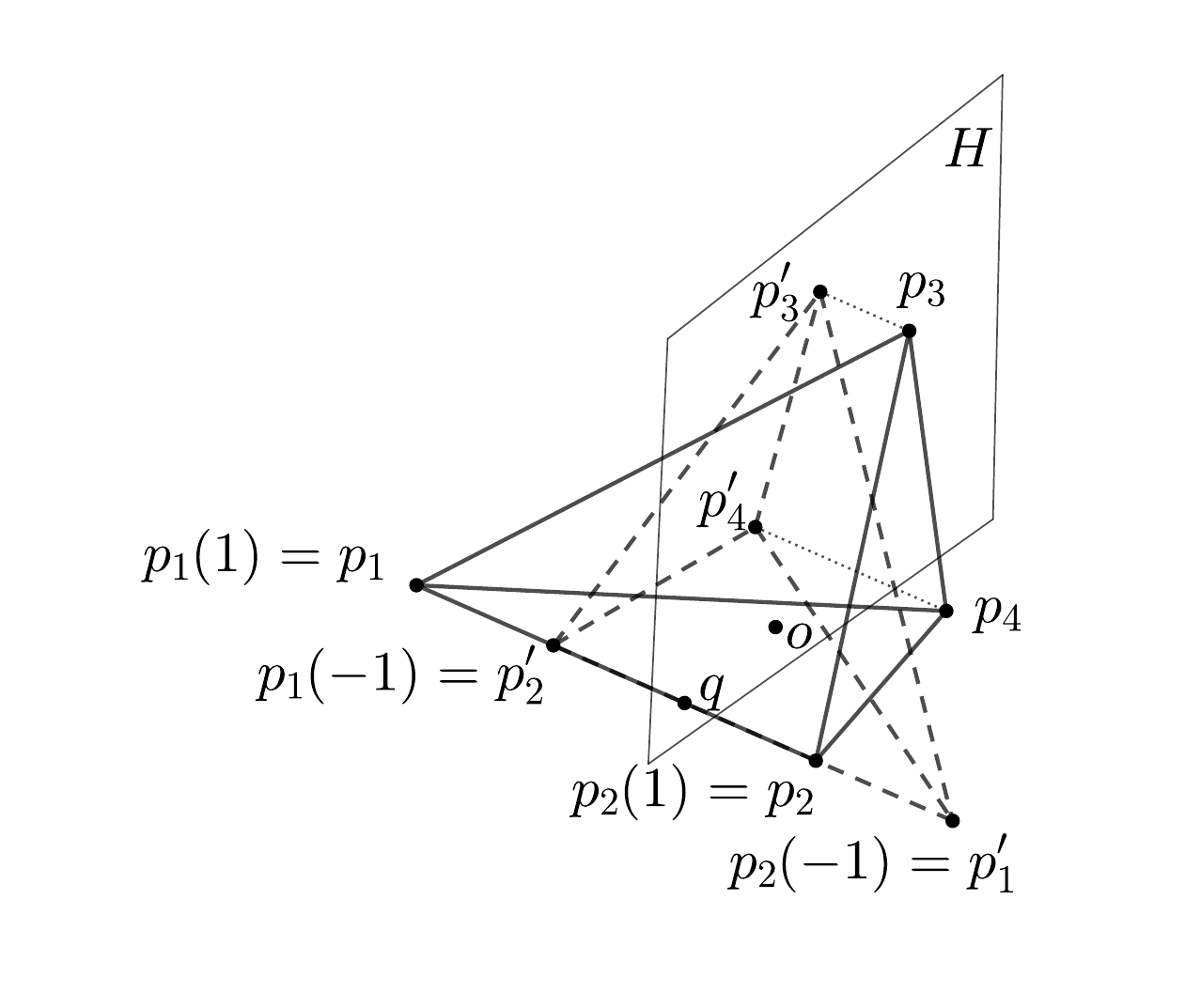}
  \caption{Lemma 2 in case $d=3$.}
\end{figure}
Without loss of generality, assume that $S$ is not symmetric to $H$, and let $S'$ denote the reflected copy of $S$ about $H$.
Let $v$ be a unit normal vector of $H$ and let $\delta_i$ denote the signed distance of $p_i$ from $H$, where the sign is chosen in such a way that distances are positive in the open half space bounded by $H$ and containing $v$. Note that for any value of $i$, $p_i - \delta_i v$ is the orthogonal projection of $p_i$ onto $H$, which we denote by $q_i$, and set $q=q_1=q_2$.
For any $i \geq 3$, let $p_i(t) = q_i + \delta_i t v$, and let $p_1(t) = q + \frac{\delta_1 - \delta_2}{2} v + \frac{\delta_1 + \delta_2}{2} t v$ and
$p_2(t) = q + \frac{\delta_2-\delta_1}{2} v + \frac{\delta_1+\delta_2}{2} t v$. Finally, let $S(t) = \conv \{ p_1(t), \ldots, p_{d+1}(t) \}$.
Clearly, $S(1)=S$, $S(-1)=S'$, and $S(0)$ is the Steiner symmetral of $S$ with respect to $H$, and hence, it is symmetric to $H$.
By Lemma~\ref{lem:shadow}, $t \mapsto f_k(S(t))$ is a strictly convex even function on $[-1,1]$, and hence, it attains a unique minimum at $t=0$. This yields the assertion for $f_k(S)$.

Finally, we consider {\zl the function $g$. Let $\mathcal{F}_{\min}$ denote the subfamily of the family of unit volume simplices that minimize $g$ in this family.} Recall that $f_1(S) = \sum_{i=1}^{d+1} |p_i|$.  Note that $f_1$ attains its minimal value in $\mathcal{F}_{\min}$. Thus, to prove the assertion it is sufficient to show that if $S$ is not a unit volume regular simplex centered at $o$, then there is a unit volume simplex $S'$ such that $g(S') < g(S)$, or $g(S')=g(S)$ and $f_1(S') < f_1(S)$. To do it we use the notation of the previous paragraph. Note that the function $t \mapsto  |\sum_{i=1}^{d+1} \varepsilon_i p_i(t) |$ is convex for any $(\varepsilon_1,\ldots, \varepsilon_{d+1}) \in \{ -1,1,\}^{d+1}$. Since the maximum of convex functions is convex, $g(S(t))$ is an even convex function on $[-1,1]$, implying that $g(S(0)) \leq g(S(1))$. Clearly, since $S$ is not symmetric to $H$, we also have that $f_1(S(0)) < f_1(S(1))$, which yields the statement.
\end{proof}

Now we prove Theorem~\ref{thm:volume}.

\begin{proof}[Proof of Theorem~\ref{thm:volume}]
Let $Z_{\mathrm{rd}}^{\mathrm{reg}} = \sum_{i=1}^{d+1} [o,q_i]$, and set $S_{\mathrm{reg}} = \conv \{ q_1, q_2, \ldots,q_{d+1}\}$. Then $S_{\mathrm{reg}}$ is a regular simplex centered at $o$, and by Corollary~\ref{cor:intrinsic_formula}, $V_d(S_{\mathrm{reg}}) = \frac{1}{d!} {\zl V_d(Z_{\mathrm{rd}}^{\mathrm{reg}})}$. Now, consider any rhombic dodecahedron $Z$ with $V_d(Z) = V_d(Z_{\mathrm{rd}}^{\mathrm{reg}})$. Since the statement in Theorem~\ref{thm:volume} is invariant under translating $Z$, we may assume that $Z = \sum_{i=1}^{d+1} [o,p_i]$, where $o \in S= \conv \{ p_1,p_2, \ldots, p_{d+1} \}$. Clearly, this implies that $S$ is a nondegenerate simplex with volume {\zl $V_d(S_{\mathrm{reg}})$}. Thus, applying Lemma~\ref{lem:Steiner} for $S$ and $S_{\mathrm{reg}}$ readily implies Theorem~\ref{thm:volume}.
\end{proof}

{\zl

By Theorem~\ref{thm:volume} and the convexity of the function $x \mapsto x^{p}$, where $p > 1$, we readily obtain the following.

\begin{rem}
Let $x_1, x_2, \ldots, x_n \in \Re^d$ with $d \leq n \leq d+1$. Then for all $1 \leq k \leq d$,
\[
\left( \frac{\sum_{1 \leq i_1 < \ldots < i_k \leq n} |x_{i_1} \wedge x_{i_2} \wedge \ldots \wedge x_{i_k} |}{ \binom{n}{k}}\right)^{\frac{1}{k}} \geq
\left( \frac{\sum_{1 \leq i_1 < \ldots < i_d \leq n} |x_{i_1} \wedge x_{i_2} \wedge \ldots \wedge x_{i_d} |}{ \binom{n}{d}}\right)^{\frac{1}{d}},
\]
with equality if and only if $n=d$ and the vectors are pairwise orthogonal and are of equal length, or if the linear hull of the $x_i$ is of dimension at most $k-1$.
Furthermore, for any $p > 1$, we have
\[
\left( \frac{\sum_{1 \leq i_1 < \ldots < i_k \leq d} |x_{i_1} \wedge x_{i_2} \wedge \ldots \wedge x_{i_k} |^p}{ \binom{n}{k}}\right)^{\frac{1}{pk}} \geq
|x_{1} \wedge x_{2} \wedge \ldots \wedge x_{d} |^{\frac{1}{d}},
\]
with equality if and only if the vectors are pairwise orthogonal and are of equal length, or if the linear hull of the $x_i$ is of dimension at most $k-1$.
\end{rem}

For any simplex $S \subset \Re^d$, integer $1 \leq k \leq d-1$, and real number $m \in \Re$, let $g_k^m(S)$ denote the sum of the $m$th powers of the volumes of the $k$-faces of $S$. Tanner \cite{Tanner} showed that if $2 \leq k \leq d-1$ and $m \in (0,2]$, among simplices $S$ in $\Re^n$ with $g_1^2(S)$ fixed, $g_k^m(S)$ is maximal if and only if $S$ is regular. Our last result in this subsection is Proposition~\ref{prop:Steiner}, which can be regarded as a complement of the above result of Tanner. This result follows by a natural modification of the proof of Lemma~\ref{lem:Steiner}, and the observation that the composition of convex functions is convex.

\begin{prop}\label{prop:Steiner}
Let $1 \leq k \leq d-1$ be an integer, and $m \geq 1$ be a real number. Then, for any simplex $S \subset \Re^d$ with unit volume, $g_k^m(S)$ is maximal if and only if $S$ is regular.
\end{prop}

}

\subsection{The circumradius of a zonotope with a given mean width}\label{subsec:rh_cr_mw}

Our aim in this section is to find the parallelotopes or rhombic dodecahedra $Z$ of a given mean width $w(Z)$ that have minimal circumradius.
For parallelotopes, we have the following theorem.

\begin{thm}\label{thm:para_cirr}
If $Z \in \mathcal{Z}_{\mathrm{p}}$ satisfies $w(Z) = w(Z_{\mathrm{p}}^{\mathrm{reg}})$, then $\cirr(Z) \geq \cirr(Z_{\mathrm{p}}^{\mathrm{reg}})$, with equality if and only if $Z$ is a cube.
\end{thm}

\begin{proof}
By (\ref{eq:cirr}),
\[
4 \cdot 2^d \cdot \left( \cirr(Z) \right)^2 \geq \sum_{(\varepsilon_1, \ldots, \varepsilon_d) \in \{ -1,1\}^d} \left| \sum_{i=1}^d \varepsilon_i p_i  \right|^2 =
2^d \sum_{i=1}^d | p_i|^2,
\]
with equality if and only if the vectors $p_i$ are pairwise orthogonal. On the other hand,
\[
\sqrt{\frac{\sum_{i=1}^d | p_i|^2}{d}} \geq \frac{\sum_{i=1}^d |p_i|}{d},
\]
with equality if and only if all $p_i$ are of equal length. By Remark~\ref{rem:meanwidth}, this yields Theorem~\ref{thm:para_cirr}.
\end{proof}

Next, we examine $\mathcal{Z}_{\mathrm{rd}}$. We prove that in the family of rhombic dodecahedra having a centered canonical form, the ones with a given mean width and minimal circumradius are regular, whereas in the family of all rhombic dodecahedra this statement is not true. Recall that a zonotope $\sum_{i=1}^n [o,p_i]$ is given in a centered canonical form if $\sum_{i=1}^n p_i = o$.

We start with a simple observation.

\begin{rem}\label{rem:regular_cirr}
Let $Z_{\mathrm{rd}}^{\mathrm{reg}} = \sum_{i=1}^{d+1} [o,q_i]$ with $q_i \in \Sph^{d-1}$ for all values of $i$. Then
\[
\cirr(Z_{\mathrm{rd}}^{\mathrm{reg}})=\left\{
\begin{array}{cc}
  \frac{\sqrt{d+2}}{2}&\text{if $d$ is even,} \\
  \frac{d+1}{2\sqrt{d}}&\text{if $d$ is odd.} \\
\end{array}\right.
\]
Furthermore, for some $\varepsilon_i \in \{ -1,1\}$, $i=1,2,\ldots,d+1$ we have $\frac{1}{2} | \sum_{i=1} \varepsilon_i q_i | = \cirr(Z_{\mathrm{rd}}^{\mathrm{reg}})$ if and only if the number of the $\varepsilon_i$ equal to $-1$ is $\lfloor \frac{d+1}{2} \rfloor$ or $\lceil \frac{d+1}{2} \rceil$.
\end{rem}

\begin{proof}
Clearly, the value of $| \sum_{i=1}^{d+1} \varepsilon_i p_i |$ depends only on the numbers of the positive and negative coefficients. Thus, we may assume that
there is some $0 <k \leq d+1$ such that $\varepsilon_1 = \ldots = \varepsilon_k = 1$ and $\varepsilon_{k+1} = \ldots = \varepsilon_{d+1} = -1$. Furthermore, observe that $\sum_{i=1}^{d+1} p_i = o$ implies that $\frac{1}{2} | \sum_{i=1}^k p_i - \sum_{i=k+1}^{d+1} p_i | = |\sum_{i=1}^k p_i |$.
By the same fact and since the value of $\langle p_i, p_j \rangle$ is independent of $i,j$ for all $i \neq j$, we also have $\langle p_i,p_j\rangle=-\frac{1}{d}$ for any $i \neq j$.
Thus, an elementary computation yields that
\[
|p_1+\ldots+p_k|^2=\langle p_1+\ldots+p_k,p_1+\ldots+p_k\rangle = k + 2 \binom{k}{2} \left( - \frac{1}{d} \right) = \frac{k(d+1-k)}{d}.
\]
As $k$ is an integer, from this we have that the maximum of $|p_1+\ldots+p_k|$ is attained if and only if $k= \lfloor \frac{d+1}{2} \rfloor$ or $k=\lceil \frac{d+1}{2} \rceil$.
\end{proof}

We note that Theorem~\ref{thm:rh_cirr} in the special case that $d$ is even and all $p_i$ are unit vectors is proved in \cite[Theorem 6]{AM} based on a different idea.

\begin{thm}\label{thm:rh_cirr}
Let $d \geq 2$ and let $Z_{\mathrm{rd}}^{\mathrm{reg}} = \sum_{i=1}^{d+1} [o,q_i]$. Then for any $Z = \sum_{i=1}^{d+1} [o,p_i] \in \mathcal{Z}_{\mathrm{rd}}$ having a centered canonical form with $w(Z)=w(Z_{\mathrm{rd}}^{\mathrm{reg}})$, we have
\begin{equation}\label{eq:eq:rh_cirr}
\cirr(Z) \geq \cirr(Z_{\mathrm{rd}}^{\mathrm{reg}}),
\end{equation}
with equality if and only if $Z$ is congruent to $Z_{\mathrm{rd}}^{\mathrm{reg}}$. Furthermore, if $d$ is odd, then there is a rhombic dodecahedron $Z' = \sum_{i=1}^{d+1} [o,p'_i]$ with $w(Z') = w(Z_{\mathrm{rd}}^{\mathrm{reg}})$ and $\cirr(Z') < \cirr(Z_{\mathrm{rd}}^{\mathrm{reg}})$.
\end{thm}

\begin{proof}
Let $Z= \sum_{i=1}^{d+1} [o,p_i]$ with $w(Z)=w(Z_{\mathrm{rd}}^{\mathrm{reg}})$. Without loss of generality, we assume that $\sum_{i=1}^{d+1} |q_i| = d+1$, which by Remark~\ref{rem:meanwidth} implies that $\sum_{i=1}^{d+1} |p_i| = d+1$.
First, we prove the inequality $\cirr(Z) \geq \cirr(Z_{\mathrm{rd}}^{\mathrm{reg}})$ for any $Z \in \mathcal{Z}_{\mathrm{rd}}$ with $w(Z)=w(Z_{\mathrm{rd}}^{\mathrm{reg}})$ and having a centered canonical form.
To do it, we prove an inequality valid for any rhombic dodecahedron with $w(Z)=w(Z_{\mathrm{rd}}^{\mathrm{reg}})$. On the other hand, for simplicity we derive this formula only in the special case that $d$ is odd; if $d$ is even, the same argument can be applied with a slightly different computation.
In the proof, we denote the coordinates of any $\varepsilon \in \{ -1,1 \}^{d+1}$ by $\varepsilon = (\varepsilon_1,\ldots,\varepsilon_{d+1})$, and set $\mathcal{E} = \left\{ \varepsilon \in \{ -1,1\}^{d+1} : \sum_{i=1}^{d+1} \varepsilon_i = 0 \right\}$.

Assume that $d = 2m-1$. Then we clearly have
\begin{equation}\label{eq:rh_cirr1}
2 \cirr(Z) \geq \max \left\{ \left| \sum_{i=1}^{d+1} \varepsilon_i p_i \right| : \varepsilon \in \mathcal{E} \right\} \geq \sqrt{ \frac{1}{\binom{d+1}{m}} \sum_{\varepsilon \in \mathcal{E}} \left| \sum_{i=1}^{d+1} \varepsilon_i p_i \right|^2 }.
\end{equation}
Furthermore,
\begin{equation}\label{eq:rh_cirr2}
\sum_{\varepsilon \in \mathcal{E}} \left| \sum_{i=1}^{d+1} \varepsilon_i p_i \right|^2 = \sum_{\varepsilon \in \mathcal{E}} \left\langle \sum_{i=1}^{d+1} \varepsilon_i p_i, \sum_{i=1}^{d+1} \varepsilon_i p_i \right\rangle = \binom{d+1}{m}\sum_{i=1}^{d+1} |p_i|^2 + 2 \sum_{\varepsilon \in \mathcal{E}} \sum_{1 \leq i < j \leq d+1} \varepsilon_i \varepsilon_j \langle p_i, p_j \rangle.
\end{equation}
An elementary computation shows that for any $i < j$, the number of the elements $\varepsilon \in \mathcal{E}$ with the property that $\varepsilon_i = \varepsilon_j$ is $2 \binom{2m-2}{m-2}$ and the number of elements with $\varepsilon_i = - \varepsilon_j$ is
$2 \binom{2m-2}{m-1}$.
As $2 \binom{2m-2}{m-2} - 2 \binom{2m-2}{m-1} = -\frac{1}{2m-1} \binom{2m}{m} = - \frac{1}{d} \binom{d+1}{m}$, this implies that
\begin{equation}\label{eq:rh_cirr3}
\sum_{\varepsilon \in \mathcal{E}} \sum_{1 \leq i < j \leq d+1} \varepsilon_i \varepsilon_j \langle p_i, p_j \rangle = -\frac{1}{d} \binom{d+1}{m} \sum_{1 \leq i < j \leq d+1} \langle p_i, p_j \rangle.
\end{equation}
On the other hand,
\[
\sum_{1 \leq i < j \leq d+1} \langle p_i, p_j \rangle = \frac{1}{2} \left( \left\langle \sum_{i=1}^{d+1} p_i, \sum_{j=1}^{d+1} p_j \right\rangle - \sum_{i=1}^{d+1} |p_i|^2 \right) = \frac{1}{2} \left( |p|^2 - \sum_{i=1}^{d+1} |p_i|^2 \right),
\]
where $p = \sum_{i=1}^{d+1} p_i$.
Substituting it back into (\ref{eq:rh_cirr3}) and applying (\ref{eq:rh_cirr2}) and (\ref{eq:rh_cirr1}), we obtain that
\[
2 \cirr(Z) \geq \sqrt{ \frac{d+1}{d} \sum_{i=1}^{d+1} |p_i|^2 - \frac{1}{d} |p|^2 }.
\]

Now, if $Z$ is given in a centered canonical form, then $p=o$, which implies, together with the inequality between the arithmetic and the quadratic means, that
\[
\cirr(Z) \geq \frac{1}{2} \sqrt{\frac{d+1}{d} \sum_{i=1}^{d+1} |p_i|^2 } \geq \frac{d+1}{2 \sqrt{d}} \cdot \frac{\sum_{i=1}^{d+1} |p_i|}{d+1} \geq \frac{d+1}{2 \sqrt{d}}.
\]
This, combined with Remark~\ref{rem:regular_cirr}, yields the inequality in (\ref{rem:regular_cirr}). Assume that we have equality in (\ref{rem:regular_cirr}). Then, by the inequality between the arithmetic and the quadratic means, $|p_i| = 1$ for all values of $i$. Furthermore, by (\ref{eq:rh_cirr3}), for any $\varepsilon \in \mathcal{E}$ we have $\left| \sum_{i=1}^{d+1} \varepsilon_i p_i \right| = \frac{d+1}{ \sqrt{d}}$, or equivalently, for any subset $J \subset \{ 1,2,\ldots, d+1 \}$ consisting of $m$ indices, we have that $\left| \sum_{j \in J} p_j \right| = \frac{d+1}{ 2\sqrt{d}}$.

Consider an arbitrary $J \subset \{ 1,2, \ldots, d+1 \}$ such that $|J| = m-1$ and $1,2, \notin J$. Then, by the previous observation, $\left| p_1 + \sum_{j \in J} p_j \right|^2 = \left| p_2 + \sum_{j \in J} p_j \right|^2$. Computing both sides and using the fact that $|p_1|=|p_2|=1$, we obtain that $\left\langle p_1, \sum_{j \in J} p_j \right\rangle = \left\langle p_2, \sum_{j \in J} p_j \right\rangle$, which yields also that $\left| p_1 - \sum_{j \in J} p_j \right| = \left| p_2 - \sum_{j \in J} p_j \right|$. Thus, $\sum_{j \in J} p_j$ lies in the bisector $H_{12}$ of the segment $[p_1,p_2]$. Note that $o \in H_{12}$ follows from the fact that $|p_1|=|p_2|$. Observe that our conditions imply that for any $i,j \geq 3$, we have $p_i - p_j \in H_{12}$. Thus, if some $p_j$ with $j \geq 3$, say $p_3$ does not lie in $H_{12}$, then all $p_j$ lie in the hyperplane $p_3 + H_{12}$. This yields that for any $J$ satisfying the conditions, $\sum_{j \in J} p_j \in (m-1)p_3 + H_{12} \neq H_{12}$, a contradiction. Hence, we have that $p_j \in H_{12}$ for all $j \geq 3$, implying that $\conv \{ p_1, p_2, \ldots, p_{d+1} \}$ is symmetric to the hyperplane $H_{12}$. Repeating the same argument for any pair of indices in place of $1$ and $2$, we have that $S=\conv \{ p_1, p_2, \ldots, p_{d+1} \}$ is symmetric to any hyperplane through $o$ and a $(d-2)$-face of $S$. This yields that $S$ is a regular simplex centered at $o$, or in other words, $Z$ is a regular rhombic dodecahedron.

Now we show that for $d$ odd, if we drop the condition that $Z$ is centered, then, among rhombic dodecahedra with a given mean width, non-regular ones have minimal circumradius. Consider the regular rhombic dodecahedron $Z_{\mathrm{rd}}^{\mathrm{reg}} = \sum_{i=1}^{d+1} [o,q_i]$, and let $v \neq o$ be an arbitrary vector. Let
$\bar{Z} = \sum_{i=1}^{d+1} [o,q_i+v]$. By Remark~\ref{rem:regular_cirr}, if $|v|$ is sufficiently small, then
\[
\cirr(\bar{Z}) = \left| \sum_{i=1}^{d+1} \varepsilon_i (q_i+v) \right|, \varepsilon \in \{ -1,1\}^{d+1}
\]
implies that $\sum_{i=1}^{d+1} \varepsilon_i = 0$. As $d$ is odd, from this we have that $\sum_{i=1}^{d+1} \varepsilon_i (q_i+v) = \sum_{i=1}^{d+1} \varepsilon_i q_i$, which yields that $\cirr (\bar{Z}) = \cirr(Z_{\mathrm{reg}})$.

We show that $w(\bar{Z}) > w(Z_{\mathrm{rd}}^{\mathrm{reg}})$. To do it it is sufficient to show that the unique point $x \in S_{\mathrm{reg}} = \conv \{ q_1, q_2, \ldots, q_{d+1} \}$ for which the function $\sum_{i=1}^{d+1} |x-q_i|$ is minimal is $o$. To prove it, let $h$ denote the length of the altitudes of $S_{\mathrm{reg}}$, and for $i=1,2,\ldots,d+1$, let $h_i$ denote the distance of $x$ from the facet of $S_{\mathrm{reg}}$ opposite of $q_i$. It is well known that for any $x \in S_{\mathrm{reg}}$, $\sum_{i=1}^{d+1} h_i = h$. Thus, for any $x \in S_{\mathrm{reg}}$, we have
\[
\sum_{i=1}^{d+1} |x-q_i| \geq \sum_{i=1}^{d+1} (h-h_i) = dh,
\]
with equality if and only if $x$ lies on all altitudes of $S_{\mathrm{reg}}$, or in other words, if $x=o$.

We have shown that $\cirr (\bar{Z}) = \cirr(Z_{\mathrm{rd}}^{\mathrm{reg}})$ and $w(\bar{Z}) > w(Z_{\mathrm{reg}})$. Thus, if we set $t = \frac{w (Z_{\mathrm{rd}}^{\mathrm{reg}}) }{w(\bar{Z})}$, then for $Z'=tZ$ we have $w(Z') = w(Z_{\mathrm{rd}}^{\mathrm{reg}})$ and $\cirr(Z') < \cirr(Z_{\mathrm{rd}}^{\mathrm{reg}})$, which proves the last statement of Theorem~\ref{thm:rh_cirr}.
\end{proof}

\subsection{The second intrinsic volume of a zonotope with a given mean width}\label{subsec:rh_mw_V2}

We start with a simple observation for parallelotopes.

\begin{prop}\label{prop:second_p}
For any $Z \in \mathcal{Z}_{\mathrm{p}}$ with $w(Z) = w(Z_{\mathrm{p}}^{\mathrm{reg}})$, we have $V_2(Z) \leq V_2(Z_{\mathrm{p}}^{\mathrm{reg}})$, with equality if and only if $Z$ is a cube.
\end{prop}

\begin{proof}
Let $Z = \sum_{i=1}^d [o,p_i]$, and set $\lambda_i = | p_i |$ {where, without loss of generality, we may assume that all the $\lambda_i$ are positive.}. Then, for any $i \neq j$, we have $|p_i \wedge p_j| \leq \lambda_i \lambda_j$, with equality if and only $p_i$ and $p_j$ are orthogonal. Thus, we need to show that if $\sum_{i=1}^d \lambda_i$ is fixed, then $A=\sum_{1 \leq i < j \leq d} \lambda_i \lambda_j$ is maximal if and only if all $\lambda_i$ are equal. But this follows {\zl from
the Maclaurin inequality (see Lemma~\ref{lem:Maclaurin})}.
\end{proof}

In the remaining part of Subsection~\ref{subsec:rh_mw_V2}, we find the equilateral rhombic dodecahedra with a given mean width and maximal second intrinsic volume. Our proof is a modification of a proof of \cite[Theorem 3]{AM}. {\zl We remark that the `dual' problem, that is, minimizing the $i$th intrinsic volume of a zonotope with a fixed mean width was considered in \cite{Hug} for the special class of zonotopes with an isotropic generating measure.}

\begin{thm}\label{thm:rh_v2}
Let $Z_{\mathrm{rd}}^{\mathrm{reg}}= \sum_{i=1}^{d+1} [o,q_i]$, where $q_i \in \Sph^{d-1}$ for all values of $i$.
Then, if $Z=\sum_{i=1}^{d+1} [o,p_i]$ is a rhombic dodecahedron with $p_i \in \Sph^{d-1}$ for all values of $i$, then
\[
V_2(Z) \geq V_2(Z_{\mathrm{rd}}^{\mathrm{reg}}),
\]
with equality if and only if $Z$ is regular.
\end{thm}

\begin{proof}
It is well known that for any vectors $x_1, x_2, \ldots, x_k$ in $\Re^d$ with $k \leq d$, the square of the $k$-dimensional volume of the parallelotope spanned by $x_1, x_2, \ldots, x_k$ is equal to the determinant of the Gram matrix of the vectors (see e.g. \cite{GoverKrikorian}). In particular, for any $x_1, x_2 \in \Re^d$, the square of the area of the parallelogram spanned by $x_1, x_2$ is
\[
|x_1 \wedge x_2 |^2 = \det ([x_1,x_2]^T [x_1,x_2]) = \langle x_1, x_1 \rangle \langle x_2, x_2 \rangle - \langle x_1, x_2 \rangle^2 .
\]
This and Corollary~\ref{cor:intrinsic_formula} imply that
\[
V_2(Z) = \sum_{1 \leq i < j \leq d+1} \sqrt{\langle p_i, p_i \rangle \langle p_j, p_j \rangle - \langle p_i, p_j \rangle^2}.
\]
We note that this yields, in particular, that $V_2(Z_{\mathrm{reg}}) = \binom{d+1}{2} \cdot \sqrt\frac{d^2-1}{d^2} = \frac{1}{2}(d+1) \sqrt{d^2-1}$.
By applying the fact that $Z$ is of unit edge length and also the inequality between the arithmetic and the quadratic means we obtain that
\begin{equation}\label{eq:2nd_intrinsic_1}
 \left( V_2(Z) \right)^2 \leq \binom{d+1}{2} \left( \binom{d+1}{2} - \sum_{1 \leq i < j \leq d+1} \langle p_i, p_j \rangle^2 \right),
\end{equation}
with equality if and only if $Z$ is regular.
In the following part we determine the minimum value of $\sum_{1 \leq i < j \leq d+1} \langle p_i, p_j \rangle^2$.

Since the vectors $p_1, p_2, \ldots, p_{d+1}$ are linearly dependent, there are constants $\lambda_i \in \Re$, where $i=1,2,\ldots,d+1$, such that
$\sum_{i=1}^{d+1} \lambda_i p_i = o$ and {\zl $\sum_{i=1}^{d+1} \lambda_i^2 > 0$}. Without loss of generality, we may assume that {\zl $\sum_{i=1}^{d+1} \lambda_i^2=1$}.
Multiplying both sides of this vector equation by itself and using the equalities $\langle p_i, p_i \rangle = 1$, we obtain that
\[
0 = \sum_{i=1}^{d+1} \lambda_i^2 + 2 \sum_{1 \leq i < j \leq d+1} \lambda_i \lambda_j \langle p_i, p_j \rangle {\zl = 1+  2 \sum_{1 \leq i < j \leq d+1} \lambda_i \lambda_j \langle p_i, p_j \rangle}.
\]
Thus, by the Cauchy-Schwarz Inequality, we have
\[
{\zl 1} = - 2 \sum_{1 \leq i < j \leq d+1} \lambda_i \lambda_j \langle p_i, p_j \rangle \leq 2 \sqrt{\sum_{1 \leq i < j \leq d+1} \lambda_i^2 \lambda_j^2} \sqrt{\sum_{1 \leq i < j \leq d+1} \langle p_i, p_j \rangle^2},
\]
implying that
\begin{equation}\label{eq:2nd_intrinsic_2}
\sum_{1 \leq i < j \leq d+1} \langle p_i, p_j \rangle^2 \geq \frac{ {\zl 1} }{4 \sum_{1 \leq i < j \leq d+1} \lambda_i^2 \lambda_j^2},
\end{equation}
{\zl which yields by Lemma~\ref{lem:Maclaurin} that
\[
\sum_{1 \leq i < j \leq d+1} \langle p_i, p_j \rangle^2 \geq \frac{2d}{d+1}.
\]
}
%On the other hand, by the inequality between the arithmetic and quadratic means, we have
%\[
%\frac{\left( \sum_{i=1}^{d+1} \lambda_i^2 \right)^2}{4 \sum_{1 \leq i < j \leq d+1} \lambda_i^2 \lambda_j^2} = \frac{\left( \sum_{i=1}^{d+1} \lambda_i^2 \right)^2}{2 \left( \left( \sum_{i=1}^{d+1} \lambda_i^2 \right)^2 - \sum_{i=1}^{d+1} \lambda_i^4 \right)} =
%\frac{1}{2 \left( 1 - \frac{1}{\left( \sum_{i=1}^{d+1} \lambda_i^2 \right)^2}\right)} \geq \frac{1}{2 \left( 1 - \frac{1}{d+1} \right)} = \frac{d+1}{2d}.
%\]
Combining this inequality {\zl with (\ref{eq:2nd_intrinsic_1})}, we obtain that
\[
V_2(Z) \leq \sqrt{\binom{d+1}{2}\left(\binom{d+1}{2} - \frac{d+1}{2d}\right)} = \frac{1}{2}(d+1)\sqrt{d^2-1} = V_2(Z_{\mathrm{reg}}),
\]
with equality if and only if $Z$ is regular.
\end{proof}

\subsection{The squared $k$-volumes of a zonotope}\label{subsec:rh_squared}

{\zl Recall that for any simplex $S \subset \Re^d$, integer $1 \leq k \leq d-1$, and real number $m \in \Re$, $g_k^m(S)$ denotes the sum of the $m$th powers of the volumes of the $k$-faces of $S$ \cite{Tanner}.} Similarly, Filliman \cite{Filliman} investigated the  problem of finding the zonotopes of maximal volume with a fixed value of the squared lengths of its generating vectors. These results, Corollary~\ref{cor:intrinsic_formula}, and the results {\zl of Brazitikos and McIntyre in \cite{BM} is the motivation behind our next definition}.

\begin{defn}\label{defn:alpha}
Let $Z = \sum_{i=1}^n [o,p_i]$ be a zonotope in $\Re^d$ and let {\zl $\alpha > 0$}. We call the quantity
\[
V_{k,\alpha} (Z) = \sum_{1 \leq i_1 < i_2 < \ldots < i_k \leq n} V_k^{\alpha}(P(i_1, \ldots, i_k))
\]
the \emph{total $k$-volume of power $\alpha$} of $Z$.
\end{defn}

Our next result is an Alexandrov-Fenchel type inequality for the total squared $k$-volumes of rhombic dodecahedra. We remark that an analogous statement for parallelotopes can be found in \cite{BM}.
%This result is a variant of Theorem 2 in \cite{BM} for $d+1$ vectors.

\begin{thm}\label{thm:squared}
Let $Z$ be a rhombic dodecahedron in $\Re^d$.
Then, for any $1 \leq k < m \leq d$, the quantity
\[
\frac{\left( V_{k,2} (Z) \right)^m }{\left( V_{m,2}(Z) \right)^k}
\]
is minimal if and only if $Z$ is regular.
\end{thm}

First, we prove a slight generalization of {\zl Lemma~\ref{lem:Maclaurin}}.

\begin{lem}\label{lem:Maclaurin_v2}
Let $1 \leq k < m$ be integers, and $x_1, \ldots, x_m \geq 0$ be nonnegative real numbers. Then
\begin{equation}\label{eq:Maclaurin}
\left( \frac{\sigma_m^k(x_1,x_2,\ldots,x_m)}{\binom{m}{k}} \right)^{\frac{1}{k}} \geq \left( \frac{\sigma_m^\textbf{k+1}(x_1,x_2,\ldots,x_m)}{\binom{m}{k+1}} \right)^{\frac{1}{k+1}},
\end{equation}
with equality if and only if all $x_i$ are equal, or at least $m-k+1$ of them is zero.
\end{lem}

\begin{proof}
Since both sides in (\ref{eq:Maclaurin}) are continuous functions of their variables, the inequality in (\ref{eq:Maclaurin}) clearly holds. Similarly, the equality case in Lemma~\ref{lem:Maclaurin_v2} holds if all $x_i$ are positive. Note that for any values of $m,k, x_1, x_2, \ldots, x_m$, $\sigma_m^k(x_1,\ldots,x_m) = \sigma_{m+1}^k(x_1,\ldots,x_m,0)$. Thus, to prove Lemma~\ref{lem:Maclaurin_v2} for all nonnegative $x_i$, it is sufficient to show that if $\sqrt[k]{\frac{A}{\binom{m}{k}}} \geq \sqrt[k+1]{\frac{B}{\binom{m}{k+1}}}$ for some reals $A,B > 0$ and integers $1 \leq k < m$, then $\sqrt[k]{\frac{A}{\binom{m+1}{k}}} > \sqrt[k+1]{\frac{B}{\binom{m+1}{k+1}}}$. In other words, we need to show that
\begin{equation}\label{eq:Maclaurin2}
\left( \frac{ \binom{m}{k} }{\binom{m+1}{k}} \right)^{\frac{1}{k}} >
\left( \frac{ \binom{m}{k+1} }{\binom{m+1}{k+1}} \right)^{\frac{1}{k+1}}
\end{equation}
for all $1 \leq k < m$. An elementary computation shows that the inequality in (\ref{eq:Maclaurin2}) is equivalent to
\[
\left( 1 - \frac{k}{m+1} \right)^{\frac{1}{k}} > \left( 1 - \frac{k+1}{m+1} \right)^{\frac{1}{k+1}},
\]
and thus, the assertion follows from the fact that for any $0 < a < 1$, the function $x \mapsto \left( 1-ax \right)^{\frac{1}{x}}$ is strictly decreasing on the interval $\left( 0, \frac{1}{a} \right]$.
\end{proof}

\begin{proof}[Proof of Theorem~\ref{thm:squared}]
Consider the zonotope $Z = \sum_{i=1}^{d+1} [o,p_i]$. Let $G$ be the Gram matrix of the vectors $p_i$; i.e. $G = [\langle p_i, p_j \rangle]$. Then $G$ is a $(d+1) \times (d+1)$ symmetric, positive semidefinite matrix. Let $p(\lambda) = \det (G - \lambda E)$ denote the characteristic polynomial of $G$. Clearly, the coefficient of $\lambda^{d+1-k}$ is equal to the signed sum of the principal minors of size $k \times k$ \cite{Meyer}, where a principal minor of size $k \times k$ of $G$ is the determinant of a $k \times k$ submatrix of $G$ symmetric to the main diagonal of $G$.
On the other hand, it is also well known that this coefficient is equal to the $k$th elementary symmetric function of the eigenvalues of $G$. Note that the rank of $G$ is equal to the dimension of the linear hull of the vectors $v_i$, and thus, $0$ is an eigenvalue of $G$ with multiplicity at least one.
Thus, the inequality in Theorem~\ref{thm:squared}, together with the equality part, follows from Lemma~\ref{lem:Maclaurin_v2} applied for $d$ eigenvalues of $G$, containing all nonzero ones.
\end{proof}

{\zl

We also note that a straightforward modification of the proof of Theorem~\ref{thm:squared} verifies Conjecture~\ref{conj:Maclaurin} for $p=2$ in the following form.

\begin{thm}\label{thm:squared2}
Let $n \geq d$ and $Z = \sum_{i=1}^{n} [o,x_i]$ be a zonotope in $\Re^d$.
Then, for any $1 \leq k < d$, the quantity
\[
\left( \frac{V_{k,2} (Z)}{ \binom{n}{k} } \right)^{\frac{1}{k}} \geq \left( \frac{V_{k+1,2} (Z)}{ \binom{n}{k+1} } \right)^{\frac{1}{k+1}},
\]
with equality if and only if $n=d$ and $Z$ is a cube, or if the dimension of $Z$ is at most $k-1$.
\end{thm}

}

\section{Isoperimetric problems for zonotopes generated by a large number of segments}\label{sec:asymptotic}

Recall that for any $d \geq 2$ and $n \geq d$, $\mathcal{Z}_{d,n}$ denotes the family of $d$-dimensional zonotopes generated by $n$ segments.
In this section we investigate isoperimetric problems for zonotopes generated by sufficiently many segments. It is worth noting that as the Euclidean ball can be approached arbitrarily well by zonotopes, and in Euclidean space a ball is the solution of most isoperimetric problems among all convex bodies, the problems discussed in this section are closely related to the problem of how well a Euclidean ball can be approached with a zonotope generated by a given number of segments.
The latter problem was extensively studied in the literature, motivated by the problem of determining the minimum number of directions we need to estimate the surface area of a convex body up to error $\varepsilon$ from the areas of the projections in the chosen directions. The order of magnitude of the number of directions was established in most dimensions in a series of papers \cite{BetkeMcMullen, Linhart, BL, BLM, Matousek} by Linhart, Bourgain, Lindestrauss, Milman and Matou\v{s}ek. Namely, it is known that for $d \geq 2$ there are constants $c_1=c_1(d)$ and $c_2=c_2(d)$ depending only on $d$ with the following property: if for any $\varepsilon > 0$, $N(\varepsilon)$ denotes the minimum number $n$ such that there is a zonotope $Z \in \mathcal{Z}_{d,n}$ with $\BB^d \subseteq Z \subseteq (1+\varepsilon) \BB^d$, then
\begin{equation}\label{eq:BLM}
c_1 \varepsilon^{\frac{-2(d-1)}{d+2}} \leq N(\varepsilon) \leq \left\{
\begin{array}{l}
c_2 \varepsilon^{\frac{-2(d-1)}{d+2}}, \hbox{ if } d=2 \hbox{ or } d \geq 5,\\
c_2 \left( \varepsilon^{-2} \log |\varepsilon|\right)^{\frac{(d-1)}{d+2}}, \hbox{ otherwise}.
\end{array}
\right.
\end{equation}
Furthermore, it was proved in \cite{BourgLind} that the weaker bound on the right-hand side of (\ref{eq:BLM}) can be attained with an equilateral zonotope, i.e. there is a zonotope $Z$ generated by $n \leq c_2 \left( \varepsilon^{-2} \log |\varepsilon|\right)^{\frac{(d-1)}{d+2}}$ segments of equal length such that $\BB^d \subseteq Z \subseteq (1+\varepsilon) \BB^d$.

Similarly like in the previous problem, in our investigation we regard the dimension $d$ as a fixed parameter, and the number $n$ of generating segments as a variable. Furthermore, we set
\[
U_d(n) = \left\{
\begin{array}{l}
\frac{\sqrt{\log n}}{n^{\frac{d+2}{2d-2}}}, \hbox{ if } d = 3 \hbox{ or } d=4,\\
\frac{1}{n^{\frac{d+2}{2d-2}}}, \hbox{ if } d = 2 \hbox{ or } d \geq 5 .
\end{array}
\right.
\]

Our first result is an immediate consequence of (\ref{eq:BLM}), and we omit its proof.

\begin{thm}\label{thm:ircr}
Let $d \geq 2$ be fixed. Then there are positive constants  $c=c(d)$ and $C=C(d)$ depending only on the dimension such that for any $n \geq d+1$,
\begin{equation}\label{eq:ircr}
\frac{ c }{n^{\frac{d+2}{2d-2}}} \leq \min \left\{  \frac{\cirr(Z)}{\ir(Z)}-1 : Z \in \mathcal{Z}_{d,n} \right\} \leq C U_d(n) .
\end{equation}
\end{thm}

In the remaining part of this section we investigate similar problems. More specifically, in Subsections~\ref{subsec:int_ir} and \ref{subsec:int_cr} we estimate the intrinsic volumes of a zonotope $Z \in \mathcal{Z}_{d,n}$ with a given inradius or circumradius, respectively. Finally, in Subsection~\ref{subsec:AF} we compare two intrinsic volumes of a zonotope in $\mathcal{Z}_{d,n}$. %{\color{blue} In our proofs we often rely on the well-known solution of these problems in the planar case.

%\begin{rem}\label{rem:2d}
%The elements of the set $\mathcal{Z}_{2,n}$, where $n \geq 3$, are the centrally symmetric convex polygons with at most $2n$ sides. By the Discrete Isoperimetric Inequality and Dowker's theorems, we have that $\min \{ V_2(Z) : Z \in \mathcal{Z}_{2,n}, \ir{Z} = 1 \} = \pi \cdot \frac{\tan \frac{\pi}{2n}}{\frac{\pi}{2n}}$, $\max \{ \perim(Z) : Z \in \mathcal{Z}_{2,n} , \cirr(Z) = 1 \} = 4n \sin \frac{\pi}{2n}$, and $\left\{ V_2(Z) : Z \in \mathcal{Z}_{2,n}, \perim(Z)=1 \right\}$
%\end{rem}

\subsection{Intrinsic volumes of a zonotope with a given inradius}\label{subsec:int_ir}

Our main result in this {\zl subsection} is Theorem~\ref{thm:irint}.

\begin{thm}\label{thm:irint}
Let $1 \leq i \leq d$. Then there is a positive constant $C=C(d)$ depending only on $d$ such that for any sufficiently large value of $n$, we have
\[
\frac{4 i}{5d n^2} \leq \min \left\{ \frac{V_i(Z)}{V_i(\BB^d)} - 1: Z \in \mathcal{Z}_{d,n}, \ir(Z)= 1 \right\} \leq
C U_d(n).
\]
\end{thm}

To prove it we need the following lemma.

\begin{lem}\label{lem:irvol}
Let {\zl $2 \leq d \leq n$}. Then we have
\[
\min \left\{ V_d(Z) : Z \in \mathcal{Z}_{d,n},  \ir(Z) \geq 1 \right\}  \geq \kappa_d \left( 1+ \frac{\pi^2}{12 n^2} \right).
\]
\end{lem}

\begin{proof}
First, let $d=2$. Consider any $Z \in \mathcal{Z}_{2,n}$ with $\ir(Z) \geq 1$. Without loss of generality, assume that $\BB^2 \subset Z$. Then, we clearly have that the minimum of $V_2(Z)$ is the area of the regular $(2n)$-gon circumscribed about $\BB^2$. Thus, an elementary computation shows that $V_2(Z) \geq 2n \tan \frac{\pi}{2n} = \pi \cdot \frac{\tan \frac{\pi}{2n}}{\frac{\pi}{2n}}$.
On the other hand, from the third order Taylor polynomial of the tangent function with the Lagrange remainder form, and using the fact that $(\tan x)^{(4)} > 0$ for $x \in \left( 0, \frac{\pi}{2} \right)$, it follows that $\tan x > x + \frac{1}{3}x^3$. Hence, we obtain that
\[
V_2(Z) > \kappa_2 \left( 1 + \frac{\pi^2}{12 n^2} \right).
\]

In the remaining part of the proof we apply an induction on $d$.
Let $Z \in \mathcal{Z}_{d,n}$ with $\ir(Z) \geq 1$. Assume that $\BB^d \subset Z$. For any unit vector $u \in \Sph^{d-1}$, let $Z | u^{\perp}$ denote the orthogonal projection of $Z$ onto the hyperplane through $o$ and with normal vector $u$. Note that for any $u \in \Sph^{d-1}$, $Z | u^{\perp}$ is a zonotope in $\mathcal{Z}_{d-1,n}$ satisfying $\ir(Z) \geq 1$. Thus, by the induction hypothesis, $V_{d-1}(Z | u^{\perp}) \geq \kappa_{d-1}  \left( 1+ \frac{\pi^2}{12 n^2} \right)$. Recall that by Cauchy's surface area formula \cite{Gardner}, we have
\[
\surf(Z) = \frac{1}{\kappa_{d-1}} \int_{\Sph^{d-1}} V_{d-1}(Z | u^{\perp}) \, du,
\]
where $\surf(Z)$ denotes the surface area of $Z$.
Then, applying the previous estimate and the fact that the surface area of $\Sph^{d-1}$ is $d \kappa_d$, after simplification we obtain that
\begin{equation}\label{eq:surface}
\surf(Z) \geq d\kappa_d \left( 1+ \frac{\pi^2}{12 n^2} \right).
\end{equation}
Now, let $F$ be any facet of $Z$. Then, by the convexity of $Z$ and since $\BB^d \subset Z$, the distance of $\aff(F)$ from $o$ is at least $1$. Thus, $V_d(\conv (F \cup \{o\})) \geq \frac{1}{d} V_{d-1}(F)$, implying that $\surf(Z) \leq dV_d(Z)$.
This, combined with (\ref{eq:surface}), yields the assertion.
%$V_d(Z) \geq \frac{1}{d} V_{d-1}(Z) \geq \frac{1}{2d} surf(Z) \geq \frac{1}{2} \kappa_d \left( 1+ \frac{\pi^2}{12 n^2} \right)$
\end{proof}

\begin{proof}[Proof of Theorem~\ref{thm:irint}]
Let $Z$ be an arbitrary zonotope in $\mathcal{Z}_{d,n}$ satisfying $\ir(Z)=1$. By the Alexandrov-Fenchel inequality for intrinsic volumes, we have that
\begin{equation}\label{eq:AF}
\left( \frac{V_d(Z)}{V_d(\BB^d)} \right)^{\frac{1}{d}} \leq \left( \frac{V_{d-1}(Z)}{V_{d-1}(\BB^d)} \right)^{\frac{1}{d-1}} \leq \ldots \leq \frac{V_1(Z)}{V_1(\BB^d)}.
\end{equation}
Thus,
\[
\left( \frac{V_d(Z)}{V_d(\BB^d)} \right)^{\frac{i}{d}}-1 \leq \frac{V_{i}(Z)}{V_{i}(\BB^d)}- 1 \leq \left( \frac{V_1(Z)}{V_1(\BB^d)} \right)^{i}-1.
\]
By Lemma~\ref{lem:irvol} and using that $\lim_{x \to 0} \frac{(1+x)^{\alpha}-1}{x} = \alpha$ and that $\frac{\pi^2}{12} > \frac{4}{5}$, for sufficiently large values of $n$, we have
\[
{\zl \frac{4i}{5dn^2} \leq \left( \frac{V_d(Z)}{V_d(\BB^d)} \right)^{\frac{i}{d}}-1 \leq \frac{V_{i}(Z)}{V_{i}(\BB^d)}- 1.}
\]
This yields the left-hand side inequality in Theorem~\ref{thm:irint}. To prove the right-hand side inequality, observe that if
$Z$ is a zonotope in $\mathcal{Z}_{d,n}$ with $\BB^d \subseteq Z \subseteq (1+\varepsilon) \BB^d$, then $2 \leq w(Z) \leq 2(1+\varepsilon)$. This, combined with
Theorem~\ref{thm:ircr}, yields that if $n$ is sufficiently large, then
\[
\left( \frac{V_1(Z)}{V_1(\BB^d)} \right)^{i}-1 \leq C U_d(n)
\]
for some suitably chosen $Z \in \mathcal{Z}_{d,n}$ with $\ir(Z)=1$.
\end{proof}

\subsection{Intrinsic volumes of a zonotope with a given circumradius}\label{subsec:int_cr}

We prove Theorem~\ref{thm:cirrint}.

\begin{thm}\label{thm:cirrint}
Let $1 \leq i \leq d$. Then there is a positive constant {\zl $C=C(d)$} such that for any sufficiently large value of $n$,
\[
\frac{ 2 i}{5 n^2} \leq \min \left\{ 1 - \frac{V_i(Z)}{V_i(\BB^d)} : Z \in \mathcal{Z}_{d,n}, \cirr(Z)= 1 \right\} \leq C U_d(n).
\]
\end{thm}

\begin{lem}\label{lem:cirrw}
Let $d \geq 2$ and $n \geq d$. Then, for any $Z \in \mathcal{Z}_{n,d}$ with $\cirr(Z)\geq 1$, we have
\[
1-\frac{V_1(Z)}{V_1(\BB^d)} \geq \frac{\pi^2}{24n^2} - \frac{\pi^4}{1920n^4}.
\]
\end{lem}

\begin{proof}
First, we prove the inequality for $d=2$. Let $Z \in \mathcal{Z}_{n,2}$ with $\cirr(Z)\geq 1$, and assume that $Z \subset \BB^2$. By Dowker's theorem, we have that
the perimeter of $Z$ is at most equal to the perimeter of the regular $(2n)$-gon inscribed in $\BB^2$. Thus, we have $\perim(Z) \leq 4n \sin \frac{\pi}{2n}$.
Note that $V_1(Z) = \frac{1}{2} \perim (Z)$ (cf. \cite[p.210]{Schneider}). Thus, applying the Taylor-expansion of the sine function, we have that
\[
V_1(Z) \leq \pi \left( 1 - \frac{\pi^2}{24n^2} + \frac{\pi^4}{1920n^4}\right).
\]
Since $V_1(\BB^d) = \frac{d \kappa_d}{\kappa_{d-1}}$, and in particular $V_1(\BB^2) = \pi$, this implies the assertion for $d=2$.
To prove it for $d > 2$, we recall Kubota's integral recursion formula from \cite[A.46]{Gardner}. To state it, we introduce the notation $\mathcal{G}(d,k)$ for the Grassmannian manifold of the $k$-dimensional linear subspaces of $\Re^d$, $K | S$ for the orthogonal projection of a compact, convex set $K \subset \Re^d$ to some $S \in \mathcal{d,k}$, and $\mu_{d,k}(\cdot)$ for the unique Haar probability measure of $\mathcal{G}(d,k)$.
Then this formula states that if $K$ is any compact, convex set in $\Re^d$, and $1 \leq i \leq k \leq d-1$, then
\begin{equation}\label{eq:Kubota}
V_i(K) = \frac{\binom{d}{i} \kappa_{k-i} \kappa_d}{\binom{k}{i} \kappa_{d-i} \kappa_k} \int_{\mathcal{G}_{d,k}} V_i(K | S) \, d \mu_{d,k}(S).
\end{equation}
Let $Z \in \mathcal{Z}_{d,n}$ be a zonotope with $\cirr(Z) \leq 1$. Then, for any $S \in \mathcal{G}_{d,2}$, $Z | S$ is a zonotope in $Z_{2,n}$ with $\cirr(Z | S) \leq 1$. Thus, $V_1(Z | S) \leq \pi \left( 1 - \frac{\pi^2}{24n^2} + \frac{\pi^4}{1920n^4}\right)$ for any $S \in \mathcal{G}_{d,2}$.
Now, applying (\ref{eq:Kubota}) for $Z$ with $i=1$ and $k=2$, we obtain that
\[
V_1(Z) \geq \frac{d \kappa_d}{\kappa_{d-1}} \left( 1 - \frac{\pi^2}{24n^2} + \frac{\pi^4}{1920n^4}\right),
\]
from which the assertion readily follows.
\end{proof}

\begin{proof}[Proof of Theorem~\ref{thm:cirrint}]
By the Alexandrov-Fenchel inequality for any zonotope $Z \in \mathcal{Z}_{d,n}$ with $\cirr(Z) = 1$, we have
\[
1-\left( \frac{V_d(Z)}{V_d(\BB^d)} \right)^{\frac{i}{d}} \geq 1-\frac{V_{i}(Z)}{V_{i}(\BB^d)} \geq  1-\left( \frac{V_1(Z)}{V_1(\BB^d)} \right)^{i}.
\]
Thus, the inequalities in Theorem~\ref{thm:cirrint} follow by an argument analogous to that in the proof of Theorem~\ref{thm:irint}.
\end{proof}

{\zl
\begin{rem}
It was shown in \cite[Theorem 1]{AN} (and remarked also in \cite{AM}) that
\[
M_n^1(\Sph^{d-1}) = n \mu_{d,1} + o \left( \frac{n}{\sqrt{d}}\right)
\]
if $n,d \to \infty$ and $n = \omega \left( d^2 \log d \right)$, where $\mu_{d,1} = \frac{\Gamma\left( \frac{d}{2} \right) }{\sqrt{\pi} \Gamma \left( \frac{d+1}{2} \right) }$.
Theorem~\ref{thm:cirrint} yields another estimate for this quantity. Namely, for any fixed $d \geq 2$, there is a constant $C=C(d)$ such that
\[
\left| M_n^1(\Sph^{d-1}) - n \mu_{d,1} \right| \leq C n^{\frac{d-4}{2d-2}} \sqrt{\log n} .
\]
\end{rem}

\begin{proof}
By Theorem~\ref{thm:cirrint}, there is a zonotope $Z \in \mathcal{Z}_{d,n}$ with $\cirr(Z) = 1$ satisfying
\[
1 - \frac{V_1(Z)}{V_1(\BB^d)} \leq C \frac{\sqrt{\log n}}{n^{\frac{d+2}{2d-2}}}
\]
where $C > 0$ may depend on $d$. Here, by \cite{BourgLind}, we may assume that $Z$ is an equilateral zonotope, i.e. that all generating segments of $Z$ are of length $t$ for some $t > 0$. The definition of {\zl intrinsic volume and} an elementary computation shows that $V_1(\BB^d) = \frac{d \kappa_d}{\kappa_{d-1}} = \frac{2 \sqrt{\pi} \Gamma \left( \frac{d+1}{2} \right)}{\Gamma \left( \frac{d}{2} \right)} = \frac{2}{\mu_{d,1}}$.
Let $Z' = \sum_{i=1}^n [o,p_i] = \frac{1}{t} Z$. Then $V_1(Z') = n$, and $\cirr(Z') = \frac{1}{t} = \frac{1}{2}\max_{v \in \Sph^{d-1}} \sum_{i=1}^n | \langle v,p_i \rangle|$, which implies that
\[
1 - \frac{ n\mu_{d,1} }{ 2 \cirr(Z')} \leq C \frac{\sqrt{\log n}}{n^{\frac{d+2}{2d-2}}}.
\]
From this inequality and the inequality $\frac{1}{1-x} < 1+2x$ holding for any sufficiently small $x > 0$, we obtain that
\[
2 \cirr (Z') \leq n \mu_{d,1} \left( 1 + 2 C \frac{\sqrt{\log n}}{n^{\frac{d+2}{2d-2}}} \right),
\]
which yields the assertion.
\end{proof}
}

\subsection{Comparing two intrinsic volumes of a zonotope}\label{subsec:AF}

Our main result here is Theorem~\ref{thm:intvols}.

\begin{thm}\label{thm:intvols}
Let $1 \leq i < k \leq d$. Then there are positive constants $c,C$ depending only on $d$ such that for any sufficiently large value of $n$,
\begin{equation}\label{eq:intvols}
 \frac{c}{n^{\frac{(d+2)(d+3)}{4d-4}}} \leq \min \left\{  \frac{\left( V_i(Z) \right)^{\frac{1}{i}}}{\left( V_k(Z) \right)^{\frac{1}{k}}} - \frac{\left( V_i(\BB^d) \right)^{\frac{1}{i}}}{\left( V_k(\BB^d) \right)^{\frac{1}{k}}} : Z \in \mathcal{Z}_{d,n} \right\} \leq \frac{C}{n}.
\end{equation}
Furthermore, there is a constant $\bar{c} > 0$ depending on $d$ such that
\begin{equation}\label{eq:intsurf}
\frac{\bar{c}}{n^2} \leq \min \left\{  \frac{\left( V_{d-1}(Z) \right)^{\frac{1}{{d-1}}}}{\left( V_d(Z) \right)^{\frac{1}{d}}} - \frac{\left( V_{d-1}(\BB^d) \right)^{\frac{1}{d-1}}}{\left( V_d(\BB^d) \right)^{\frac{1}{d}}} : Z \in \mathcal{Z}_{d,n} \right\} .
\end{equation}
\end{thm}

To prove this theorem, we need some preparation. In the following, let $\omega_d$ denote the surface area of the Euclidean sphere $\Sph^{d-1}$ in $\Re^d$.

\begin{rem}\label{rem:identity}
For any $d \geq 2$, we have
\[
\frac{V_d(\BB^d)}{\left( V_1(\BB^d) \right)^d} = \frac{2 \omega_{d+1}^{d-1} }{ \omega_{d}^{d} d! }.
\]
\end{rem}

\begin{proof}
Note that $\omega_d = d \kappa_d$, and $\kappa_d = \frac{\pi^{\frac{d}{2}}}{\Gamma \left( \frac{d}{2} + 1\right)}$, implying also the identity $\kappa_{d+1}=\frac{2\pi}{d+1} \kappa_{d-1}$. Furthermore, it is well known that $V_1(\BB^d ) = \frac{d \kappa_d}{\kappa_{d-1}}$. Applying these identities, we have
\[
\frac{2 \omega_{d+1}^{d-1} }{ \omega_{d}^{d} d! } \cdot \frac{ \left( V_1(\BB^d) \right)^d}{V_d(\BB^d)} =
\frac{2 (d+1)^{d-1} \kappa_{d+1}^{d-1}}{ d! \kappa_{d-1}^d \kappa_d} =
\frac{ 2^d \pi^{d-1}}{d! \kappa_d \kappa_{d-1}} =
\frac{2^d \Gamma\left( \frac{d+2}{2} \right) \Gamma\left( \frac{d+1}{2} \right) }{d! \sqrt{\pi}} =1,
\]
where in the last step we used the Legendre duplication formula $\Gamma(z) \Gamma\left( z + \frac{1}{2} \right) = 2^{1-2z} \sqrt{\pi} \Gamma (2z)$ for $z \in \mathbb{C}$, and the identity $\Gamma(k) = (k-1)!$ for all positive integers $k$.
\end{proof}

{\zl Our next lemma is a special case of \cite[Theorem 8.2.3]{SchWeil}, and hence, we omit its proof.}

\begin{lem}\label{lem:randompara}
Let $p_1, \ldots, p_d \in \Sph^{d-1}$ be unit vectors chosen independently according to the uniform probability distribution on $\Sph^{d-1}$. Then the expected value of $|p_1 \wedge p_2 \wedge \ldots \wedge p_d|$ is $\frac{2 \omega_{d+1}^{d-1} }{ \omega_{d}^{d} }$.
\end{lem}

\begin{proof}[Proof of Theorem~\ref{thm:intvols}]
First, we prove the lower bound in (\ref{eq:intvols}) for $i=1$ and $k=2$.
Let $Z \in \mathcal{Z}_{d,n}$. Since the statement is invariant under rescaling $Z$, we can assume that $V_1(Z) = V_1 (\BB^d)$ and that $Z$ is centered at the origin. Then the Steiner ball of $Z$, defined as the ball centered at the Steiner point of $Z$ and having mean width equal to that of $Z$ \cite{Schneider}, coincides with $\BB^d$. Our proof in this case is based on a stability version of the Alexandrov-Fenchel inequality, stated in (7.123) of \cite{Schneider} as a consequence of \cite[Theorem 7.6.6]{Schneider} and implying that $\frac{1}{\kappa_d} \left( W_{d-1}(Z) \right)^2 - W_{d-2}(Z) \geq \frac{d+1}{d(d-1)} \left( \delta_2 (Z, \BB^d) \right)^2$, where $W_i(K)$ denotes the $i$th quermassintegral of the convex body $K$, and $\delta_2(K,L)= \sqrt{ \int_{u \in \Sph^{d-1}} |h_K(u)-h_L(u)|^2 \, du }$, with $h_K$ and $h_L$ denoting the support functions of the convex bodies of $K$ and $L$, respectively. Applying the well known relation $W_{d-i}(Z)= \frac{\kappa_{d-i}}{\binom{d}{i}} V_i(Z)$ for $i=1,2$, and using $W_i(\BB^d) = \kappa_d$ for all values of $i$, we obtain that
\[
\frac{\left( V_1(Z) \right)^2}{\left( V_1 (\BB^d) \right)^2} - \frac{V_2(Z)}{V_2(\BB^d)} = 1 - \frac{V_2(Z)}{V_2(\BB^d)} \geq \frac{d+1}{\kappa_d d(d-1)} \left( \delta_2 (K, \BB^d) \right)^2.
\]
We note that by Lemma 7.6.5 of \cite{Schneider} (see also Lemmas 1 and 2 in \cite{Groemer}), we have
\[
\left( \delta_2 (Z, \BB^d) \right)^2 \geq c' \left( \delta (Z, \BB^d) \right)^{\frac{d+3}{2}},
\]
where $\delta(\cdot, \cdot)$ denotes the Hausdorff distance, and $c' > 0$ is a constant depending only on the dimension $d$. If for some $\varepsilon > 0$, $\delta (Z, \BB^d) \geq \varepsilon$ for all $Z \in \mathcal{Z}_{d,n}$ with Steiner ball $\BB^d$ and for all sufficiently large values of $n$, we are done. Thus, we may assume that $\delta (Z, \BB^d) \geq \frac{1}{2}$, implying that $\frac{1}{2} \leq \ir(Z)$. Observe that since the Steiner ball of $Z$ is $\BB^d$, we have $\delta (Z, \BB^d) \geq \frac{1}{2} \left( \cirr(Z) - \ir(Z) \right) \geq \frac{1}{4} \left( \frac{\cirr(Z)}{\ir(Z)} - 1 \right)$. Hence, from Theorem~\ref{thm:ircr} it follows that
\[
1 - \frac{V_2(Z)}{V_2(\BB^d)} \geq c'' \left( \delta (Z, \BB^d) \right)^{\frac{d+3}{2}} \geq \frac{c'''}{n^{\frac{(d+2)(d+3)}{4d-4}}}
\]
for some constant $c''' > 0$ depending on the dimension. After some elementary calculations, this implies the required bound.

Now we prove the lower bound in (\ref{eq:intvols}) for any value of $i$ with $k=i+1$, and observe that this clearly yields the lower bound in (\ref{eq:intvols}) for all values of $i$ and $k$. By the Alexandrov-Fenchel inequality (see e.g. \cite[Theorem 10 (ii)]{BHK}, for any $2 \leq j \leq d-1$ and any $Z \in \mathcal{Z}_{d,n}$, we have $\left( \frac{V_j(Z)}{V_j(\BB^d)} \right)^2 \geq \frac{V_{j-1}(Z)}{V_{j-1}(\BB^d)} \frac{V_{j+1}(Z)}{V_{j+1}(\BB^d)}$, implying $\frac{\left( V_j(Z)/V_j(\BB^d) \right)^{1/j}}{\left( V_{j+1}(Z)/V_{j+1}(\BB^d) \right)^{1/(j+1)}} \geq \left( \frac{\left( V_{j-1}(Z)/V_{j-1}(\BB^d) \right)^{1/(j-1)}}{\left( V_j(Z)/V_j(\BB^d) \right)^{1/j}} \right)^{\frac{j-1}{j+1}}$.
From this, it follows that
\[
\frac{\left( V_i(Z)/V_i(\BB^d) \right)^{1/i}}{\left( V_{i+1}(Z)/V_{i+1}(\BB^d) \right)^{1/(i+1)}} \geq \left( \frac{ V_1(Z)/V_1(\BB^d) }{\sqrt{ V_2(Z)/V_2(\BB^d) } } \right)^{\frac{2}{i(i+1)}}.
\]
We investigate the left-hand side. Similarly like in the previous paragraph, we may assume that $\frac{ V_1(Z)/V_1(\BB^d) }{\sqrt{ V_2(Z)/V_2(\BB^d) }} \leq 2$. Under this assumption, we obtain that
\[
\left( \frac{ V_1(Z)/V_1(\BB^d) }{\sqrt{ V_2(Z)/V_2(\BB^d) } } \right)^{\frac{2}{i(i+1)}} = \left( 1 + \frac{ V_1(Z)/\sqrt{V_2(Z)} }{V_1(\BB^d)/ \sqrt{V_2(\BB^d)} } - 1
\right)^{\frac{2}{i(i+1)}} \geq
\]
\[
\geq 1 + \frac{\sqrt{V_2(\BB^d)}}{V_1(\BB^d)} \cdot \frac{\frac{V_1(Z)}{\sqrt{V_2(Z)}}-\frac{V_1(\BB^d)}{\sqrt{V_2(\BB^d)}}}{2^{(i^2+i)/2}-1} \geq 1 + \frac{\hat{c}}{n^{\frac{(d+2)(d+3)}{4d-4}}}
\]
for some $\hat{c} > 0$ depending on $d$. From this a similar computation yields the lower bound in (\ref{eq:intvols}) for $k=i+1$.

Now we prove the inequality in (\ref{eq:intsurf}). Let $Q$ be the convex polytope with the same outer unit normals as $Z$ and circumscribed about $\BB^d$. Then, by Lindel\"of's theorem \cite{Gruber}, $\frac{\left( V_{d-1}(Z) \right)^{\frac{1}{d-1}}}{\left( V_{d}(Z) \right)^{\frac{1}{d}}} \geq \frac{\left( V_{d-1}(Q) \right)^{\frac{1}{d-1}}}{\left( V_{d}(Q) \right)^{\frac{1}{d}}}$. On the other hand, applying the volume formula of a cone for the convex hull of each facet of $Q$ with $\{ o \}$, we obtain that $\surf(Q) = d V_d(Q)$, implying $V_{d-1}(Q) = \frac{d}{2} V_d(Q)$. Thus, applying an asymptotic estimate on the minimal volume of a convex polytope with a given number of facets and circumscribed about $\BB^d$ \cite{Gruber, BHK}, we obtain that
\begin{equation}\label{eq:surfvol}
\frac{ \left( V_{d-1}(Z) \right)^{\frac{1}{d-1}} }{ \left( V_{d}(Z) \right)^{\frac{1}{d}}}  \geq
\left( \frac{d}{2} \right)^{\frac{1}{d-1}} \cdot \left( V_d(Q) \right)^{\frac{1}{d(d-1)}} \geq \left( \frac{d}{2} \right)^{\frac{1}{d-1}} \kappa_d^{\frac{1}{d(d-1)}} \cdot \left( 1+\frac{\gamma}{\left( \binom{n}{d-1} \right)^{\frac{2}{d-1}}} \right)^{\frac{1}{d(d-1)}}.
\end{equation}
for any zonotope $Z \in \mathcal{Z}_{d,n}$. Since $\left( \frac{d}{2} \right)^{\frac{1}{d-1}} \kappa_d^{\frac{1}{d(d-1)}} = \frac{\left( V_{d-1}(\BB^d) \right)^{\frac{1}{d-1}}}{\left( V_{d}(\BB^d) \right)^{\frac{1}{d}}}$, and $\binom{n}{d-1}^{\frac{2}{d-1}} \leq \frac{n^2}{((d-1)!)^{\frac{2}{d-1}}}$, this implies the bound in (\ref{eq:intsurf}).

Finally, to prove the upper bound in (\ref{eq:intvols}), first, observe that by (\ref{eq:AF}), for any $1 \leq i \leq k \leq d$, and zonotope $Z$, we have
\begin{equation}\label{eq:AF2}
\frac{\left(V_i(Z) \right)^{\frac{1}{i}} / \left( (V_i(\BB^d) \right)^{\frac{1}{i}} }{\left(V_k(Z) \right)^{\frac{1}{k}} / \left( (V_k(\BB^d) \right)^{\frac{1}{k}}} \leq \frac{V_1(Z) / V_1(\BB^d)}{\left(V_d(Z) \right)^{\frac{1}{d}} / \left(V_d(\BB^d) \right)^{\frac{1}{d}} }.
\end{equation}
We give an upper bound on $\min \left\{ \frac{V_1(Z)}{\left(V_d(Z) \right)^{\frac{1}{d}}} : Z \in \mathcal{Z}_{d,n} \right\}$.
To do it we choose $n$ unit vectors $p_i \in \Sph^{d-1}$ independently and using uniform distribution, and set $Z=\sum_{i=1}^n [o,p_i]$. Then, by the linearity of expectation and Lemma~\ref{lem:randompara}, we have that the expected value of $V_d(Z)$ is $\binom{n}{d} \cdot \frac{2 \omega_{d+1}^{d-1} }{ \omega_{d}^{d} }$.
Since $V_1(Z)=n$ for any such zonotope and by Remark~\ref{rem:identity}, there is some $Z \in \mathcal{Z}_{d,n}$ such that
\[
\frac{V_1(Z)}{\left(V_d(Z) \right)^{\frac{1}{d}}} \leq n \cdot  \left( \frac{  \omega_{d}^{d} } {2 \binom{n}{d}  \omega_{d+1}^{d-1} } \right)^{\frac{1}{d}} =
\frac{n}{\sqrt[d]{(n(n-1)\ldots (n-d+1))}} \frac{V_1(\BB^d)}{\left( V_d(\BB^d) \right)^{\frac{1}{d}}} \leq
\]
\[
\leq \left( 1 + \frac{d-1}{n-d+1} \right) \frac{V_1(\BB^d)}{\left( V_d(\BB^d) \right)^{\frac{1}{d}}} \leq \left( 1 + \frac{2d}{n} \right) \frac{V_1(\BB^d)}{\left( V_d(\BB^d) \right)^{\frac{1}{d}}},
\]
for any $n \geq 2d$.
Combining it with (\ref{eq:AF2}), we obtain that
\[
\frac{\left(V_i(Z) \right)^{\frac{1}{i}}}{\left(V_k(Z) \right)^{\frac{1}{k}}} \leq
\frac{\left( (V_i(\BB^d) \right)^{\frac{1}{i}} }{\left( (V_k(\BB^d) \right)^{\frac{1}{k}}} \cdot \left( 1 + \frac{2d}{n} \right),
\]
from which the upper bound in Theorem~\ref{thm:intvols} readily follows.
\end{proof}

\begin{rem}
We remark that if $Z_0 \in \mathcal{Z}_{d,n}$ satisfies the condition that $\ir(Z_0) = 1$ and
\[
V_i(Z_0) = \min \{ V_i(Z): Z \in \mathcal{Z}_{d,n}, \ir(Z)=1 \},
\]
then by (\ref{eq:AF}) and Theorem~\ref{thm:irint}, for any $1 \leq i < k \leq d$ we have
\[
0 \leq \left( \frac{V_i(Z)}{V_d(\BB^d)} \right)^{\frac{1}{i}} - \left( \frac{V_k(Z)}{V_k(\BB^d)} \right)^{\frac{1}{k}} \leq \left( \frac{V_i(Z)}{V_d(\BB^d)} \right)^{\frac{1}{i}} - 1 \leq CU_d(n).
\]
for some constant $C > 0$ and sufficiently large value of $n$.
\end{rem}

\section{Remarks and open problems}\label{sec:problems}

In Theorem~\ref{thm:volume}, we have shown that a regular rhombic dodecahedron has minimal circumradius among the regular dodecahedra of the same volume.

\begin{prob}
 Is it true that the circumradius of a unit volume rhombic dodecahedron $Z$ is minimal only if $Z$ is regular?
\end{prob}

\begin{prob}\label{prob:cirr}
Let $1 \leq k < d$. Find the elements $Z \in \mathcal{Z}_{d,d+1}$ with $\cirr(Z)=1$ and maximal $k$th intrinsic volume.
\end{prob}

We note that in Problem~\ref{prob:cirr}, by Theorem~\ref{thm:volume}, $Z$ is a regular rhombic dodecahedron if $k=d$, and by Theorem~\ref{thm:rh_cirr} $Z$ is not regular if $k=1$ and $d$ is odd.

\begin{prob}
Prove or disprove that for any $1 < k < d$, the elements $Z \in \mathcal{Z}_{d,d+1}$ with $w(Z)=1$ and maximal value of $V_k(Z)$ are regular rhombic dodecahedra.
\end{prob}

\begin{prob}
Find exact orders of magnitudes in the problems discussed in Section~\ref{sec:asymptotic}.
\end{prob}

\end{document}